\documentclass{article}
\usepackage{fullpage}
\usepackage{microtype,  natbib,enumitem}
\usepackage{graphicx}
\usepackage{epstopdf}
\usepackage{subfigure}
\usepackage{booktabs} 
\usepackage{graphicx}
\usepackage{amsmath}
\usepackage{amssymb}
\usepackage{color}
\usepackage{hyperref}
\usepackage{wasysym}
\usepackage{amsmath,amsfonts,color}
\usepackage{algorithm}
\usepackage[noend]{algpseudocode}

\usepackage{amsthm}
    
\newcommand{\ceil}[1]{\left \lceil{#1}\right \rceil }

\newcommand{\R}{\mathbb R}

\newcommand{\mQ}{\mathcal Q}

\newcommand{\mK}{\mathcal K}

\newcommand{\mI}{\mathcal I}

\newcommand{\mP}{\mathcal P}

\newcommand{\mG}{\mathcal G}
\newcommand{\mN}{\mathcal N}
\newcommand{\mT}{\mathcal T}
\newcommand{\mX}{\mathcal X}

\newcommand{\LMO}{\mathbf{LMO}}
\newcommand{\conv}{\mathbf{conv}}

\newcommand{\res}{\mathbf{res}}
\newcommand{\rec}{\mathbf{rec}}

\newcommand{\supp}{\mathbf{supp}}
\newcommand{\dsupp}{\mathbf{dsupp}}

\newcommand{\range}{\mathbf{range}}

\newcommand{\st}{\mathrm{st}}

\newcommand{\mb}[1]{\mathbf{#1} }

\newcommand{\minimize}[1]{\underset{#1}{\mathrm{minimize}}}
\newcommand{\maximize}[1]{\underset{#1}{\mathrm{maximize}}}
\newcommand{\mini}[1]{\underset{#1}{\mathrm{min}}}
\newcommand{\maxi}[1]{\underset{#1}{\mathrm{max}}}
\newcommand{\argmin}[1]{\underset{#1}{\mathrm{argmin}}}
\newcommand{\argsup}[1]{\underset{#1}{\mathrm{argsup}}}
\newcommand{\argmax}[1]{\underset{#1}{\mathrm{argmax}}}
\newcommand{\subjectto}{\mathrm{subject\; to}}

\newcommand{\coeff}{\mathbf{coeff}}

\newcommand{\sign}{\mathbf{sign}}
\newcommand{\gap}{\mathbf{gap}}

\newcommand{\inte}{\mathbf{int}}

\newcommand{\diam}{\mathbf{diam}}

\newcommand{\cone}{\mathbf{cone}}

\newcommand{\bmat}{\left[\begin{matrix}}
\newcommand{\emat}{\end{matrix}\right]}

\newcommand{\minmaj}{\text{(Min-Maj)}}
\newcommand{\merge}{\text{(Merge)}}
\newcommand{\psimple}{\text{(P-simple)}}
\newcommand{\dsimple}{\text{(D-simple)}}
\newcommand{\pconvex}{\text{(P-convex)}}
\newcommand{\dconvex}{\text{(D-convex)}}
\newcommand{\pgeneral}{\text{(P-general)}}
\newcommand{\dgeneral}{\text{(D-general)}}

\newtheorem{assumption}{Assumption}

\newtheorem{corollary}{Corollary}
\newtheorem{property}{Property}
\newtheorem{lemma}{Lemma}
\newtheorem{thm}{Theorem}
\theoremstyle{definition}\newtheorem{definition}{Definition}

\newtheorem{innercustomprop}{Property}
\newenvironment{customprop}[1]
  {\renewcommand\theinnercustomprop{#1}\innercustomprop}
  {\endinnercustomprop}

\begin{document}

\title{Screening  for a Reweighted Penalized Conditional Gradient Method
\thanks{
This work was funded in part by the French government under management of Agence Nationale de la Recherche as part of the ``Investissements d'avenir" program, reference ANR-19-P3IA-0001 (PRAIRIE 3IA Institute). We also acknowledge
support the European Research Council (grant SEQUOIA 724063).
The first is also funded in part  by AXA pour la recherche and Kamet Ventures, as well as a Google focused award.}
}


\author{Yifan Sun$^1$ \and Francis Bach$^2$}
\date{%
    $^1$Stony Brook University\\%
    $^2$INRIA-Paris\\[2ex]%
    October 1, 2020
}

\maketitle

\begin{abstract}
The conditional gradient method (CGM) is widely used in large-scale sparse convex optimization, having a low per iteration computational cost for structured sparse regularizers and a greedy approach to collecting nonzeros. We explore the sparsity acquiring properties of  a general penalized CGM (P-CGM) for convex regularizers and a reweighted penalized CGM (RP-CGM) for nonconvex regularizers, replacing the usual convex constraints with gauge-inspired penalties.  This generalization does not increase the per-iteration complexity noticeably. Without  assuming bounded iterates or using line search, we show $O(1/t)$ convergence of the gap of each subproblem, which measures distance to a stationary point. 
 We couple this with a screening rule which is safe in the convex case, converging to the true support at a rate $O(1/(\delta^2))$ where $\delta \geq 0$ measures how close the problem is to  degeneracy. In the nonconvex case the screening rule converges to the true support in a finite number of iterations, but is not necessarily safe in the intermediate iterates.
In our experiments, we verify the consistency of the method and adjust the aggressiveness of the screening rule by tuning the concavity of the regularizer.
\end{abstract}

\section{Introduction}
\label{sec:intro}

Conditional gradient methods (CGMs) are used in constrained optimization to quickly arrive to sparse solutions of large-scale optimization problems. In this paper, we generalize their applicability to nonconvex penalized  (unconstrained) problems and investigate safe screening methods to obtain sparse supports in finite time.
We describe these problems as
\begin{equation}
\minimize{x\in \R^d}\; f(x) + \phi(r_\mP(x))
\label{eq:main0}
\end{equation}
where $f:\R^d\to \R$ is a convex loss function with an $L$-Lipschitz continuous gradient, $\phi:\R_+\to\R$ is a strictly convex monotonically increasing function, and $r_\mP:\R^d\to\R_+$ a  nonconvex variant of a gauge function, defined as the solution to 
\begin{equation}
r_\mP(x) = \min_{c_p \geq 0} \left\{\sum_{p\in \mP_0} \gamma(c_p) p : \sum_{p\in \mP_0} c_p p  = x\right\}
\label{eq:nonconvex-gauge}
\end{equation}
for some concave monotonically increasing function $\gamma:\R_+\to\R_+$ and  $\mP_0$  a finite collection of vectors in $\R^d$. 
In the usual nonzero sparsity case, this penalty reduces to the well-studied nonconvex penalties like SCAD, LSP, or $p$-``norms" for $0<p < 1$.
Problems of this form arise in machine learning, compressed sensing, low-rank matrix factorization, etc, and are often observed in practice to be more effective sparsifiers than their convex relaxations \citep{chen2010convergence}.

In particular, we solve \eqref{eq:main0} using the following iteration scheme
\begin{align*}
s^{(t)} &= \argmin{s \in \R^d}\; \nabla f(x^{(t)})^Ts + \bar h^{(t)}(s) & \minmaj\\
x^{(t+1)} &= (1-\theta^{(t)}) x^{(t)} + \theta^{(t)} s^{(t)} & \merge,
\end{align*}
where $\bar h^{(t)}(s)$ is a convex linearization of $\phi(r_{\mP}(s))$ at $x^{(t)}$.
 We call this the reweighted penalized conditional gradient method (RP-CGM), as it  resembles both the conditional gradient method (CGM) in sparse convex optimization and reweighting schemes in majorization-minorization methods for nonconvex optimization. 
 
\paragraph{Example.} The $\ell_1$ norm is formed by picking $\mP_0 = \{\pm e_1,...,\pm e_d\}$ the signed unit bases, and $\gamma(\xi) = \xi$. Then the solution to \eqref{eq:nonconvex-gauge} is always unique and can be expressed in closed form as $r_\mP(x) =  \|x\|_1$. 
Picking instead a concave penalty $\gamma(\xi) = 2\sqrt{\xi}$ leads to the variation $r_\mP(x) = 2\sum_i \sqrt{|x_i|}$ the ``half norm".
Similar transformations also lead to the smoothed capped absolute deviation (SCAD) penalty, minimum concave penalty (MCP), etc. (See table \ref{tab:concavepenalties}.)

For the transformed 1-norm, the minimization \eqref{eq:nonconvex-gauge} leads to the same support regardless of $\gamma$; (e.g., $r_\mP(x) = \sum_i \gamma(x_i)$). 
However, this consistency of support is not true in general; an example is given in section \ref{sec:nonconvex} showing that in general, nonconvex $\gamma$ leads to \emph{sparser} support.



\paragraph{Penalization transformation $\phi$.} 
When $\phi(\xi) = \xi$ and $\gamma(c_p) = c_p$, then \eqref{eq:main-convex} is a generalization of the LASSO problem; at the other extreme, $\phi(\xi) = \iota_{\xi \leq 1}$ an indicator function can constrain $x\in \mP$, reducing  to the vanilla CGM case. The curvature of $\phi$ controls the potential growth of the variable $x$; we show that though a penalty $\phi$ is possible, a curvature requirement is needed to ensure the growth is not too large to cause divergence; at the same time, it conveniently ensures that  the Fenchel dual of \eqref{eq:main0} is unconstrained, making it easy to acquire a dual feasible point for screening.


\subsection{Related work}

\paragraph{Conditional gradient method.} When  $h(s) = \iota_{\bar\mP}(s)$ the indicator for $s$ in $\mP$, the proposed method is the conditional gradient method (CGM) \citep{frank1956algorithm,dunn1978conditional}. Also called the Frank-Wolfe method, it  has been studied since the 50s  and was revitalized recently \citep{jaggi2013revisiting} for its success at quickly estimating solutions to sparse optimization problems. 
The method is particularly  useful when computing the supporting hyperplane in the \minmaj~step is computationally cheap (e.g., when $\mP$ is the unit ball of the $\ell_1$-norm or the nuclear norm). 
Since then, much work has come from expanding its use to general (atomic) norms 
\citep{hazan2008sparse,clarkson2010coresets,jaggi2013revisiting,tewari2011greedy} 
with many variations such as backward steps \citep{lacostejulienjaggi, nikhilforwardbackward} and fully-corrective steps \citep{vonhohenbalken}. Many connections between the CGM and existing methods have also been discovered, such as to mirror descent \citep{bach2015duality}, cutting plane method \citep{limitedkelley}, and greedy coordinate-wise methods \citep{clarkson2010coresets}.
In its simplest version (with no away-steps, line search, or strongly convex assumptions on $f$) the minimum duality gap in CGM converges at rate $O(1/t)$ \citep{dunn1978conditional}.

\paragraph{Convex gauge function.} When $\gamma(c_p) = c_p$, we define $\kappa_\mP(x):=r_\mP(x)$, which is the usual convex gauge function for the closed convex set $\mP$ \citep{rockafellar1970convex,freund1987dual}.
Gauge functions can be seen as generalized versions of the $\ell_1$-norm, which is a convex promoter of nonzero vector sparsity, and include penalties like the total variation (TV) norm,  nuclear norm,  OWL norm \citep{zeng2014ordered},  OSCAR norm \citep{bondell2008simultaneous}, and general conic constraints.
Several works have looked at optimization over general gauges  \citep{friedlander2014gauge,freund1987dual} and in particular for sparse optimization \citep{chandrasekaran2012convex,jaggi2013revisiting}.

\paragraph{Penalized CGM.} When $\bar h^{(t)}(s)$ is a convex penalty, we refer to the proposed method as the penalized CGM (P-CGM). Compared to CGM,  P-CGM has been much less studied \citep{yu2017generalized,harchaoui2015conditional,mu2016scalable}, and has appeared under different names, like regularized coordinate minimization \citep{dudik2012lifted}. 
An $O(1/t)$ convergence rate has been shown for specific smooth functions \citep{mu2016scalable},  with bounded assumptions on iterates  \citep{bach2015duality}, or with improvement steps to ensure boundedness of sublevel sets \citep{yu2017generalized,harchaoui2015conditional}.
When $f$ is quadratic and for a special form of $\phi$, the P-CGM can be shown to be equivalent to a form of the iterative shrinkage method, and under proper problem conditioning, has linear convergence
\citep{bredies2009generalized,bredies2008iterated}.

\paragraph{Reweighted methods for nonconvex minimization.}
Our main algorithmic novelty is to solve a sequence of reweighted penalized CGM (RP-CGM) iterations in order to accommodate nonlinear $\gamma$, which appear in nonconvex penalties like SCAD or MCP penalties in difference-of-convex or majorization-minimization methods. This results in a nonconvex penalty $h(x)$, which complicates analysis; in practice such norms, solved via iterative reweighting methods,  have been shown to have superior sensing properties
\citep{chen2010convergence, ene2019improved,wolke1988iteratively,daubechies2010iteratively, gong2013general,ochs2015iteratively,wright2009sparse, mairal2014sparse}. We leverage these observations to improve the screening properties of RP-CGM; by increasing the concavity of $\gamma$, we can create an aggressive support recovery method based on an easily computable duality-gap-like residuals. 


\paragraph{Safe screening.}  A \emph{screening rule} returns an estimate of the support of $x^*$ given a noisy approximation $x$. The screening rule is \emph{safe} if there are no false positives (and called \emph{sure} if there are no false negatives). 
Safe screening rules for LASSO were first proposed by \citet{ghaoui2010safe}, and have since been extended to a number of smooth losses and generalized penalties 
\citep{fercoq2015mind,   liu2013safe,malti2016safe,ndiaye2015gap,wangjie,bonnefoyantoine}.  
An interesting related work is the ``stingy coordinate descent'' method \citep{johnson2017stingy} for LASSO, which optimizes the sparse regularized problem in a CGM-like manner, but uses screening to dynamically skip steps; this kind of method can be extended to P-CGM as well for generalized atoms.
In nonconvex optimization, support recovery is discussed by \citet{burke1990identification} for handling nonlinear constraints which are iteratively linearized, and screening rules by \citet{rakotomamonjy2019screening} are proposed for a reweighted proximal gradient method.


\subsection{Contributions and outline}

We analyze the support recovery and convergence properties of P-CGM and RP-CGM on  \eqref{eq:main0}.
We assume that the loss function $f$ is $L$-smooth, the function $\phi(\xi)$ grows at least asymptotically quadratically, the function $\gamma$ has slope bounded away from 0 and $+\infty$, and the set $\mP_0$ is either finite or  a union of a finite set and a nonoverlapping cone.  We give three main contributions.
\begin{itemize}

\item Under mild assumptions the RP-CGM converges to a stationary point. In particular, \emph{without boundedness assumptions on iterates}, using the deterministic step size schedule of $\theta^{(t)} = 2/(1+t)$, the function value error and gap-like residual of  RP-CGM converge as $O(1/t)$. 

\item We offer an online gap-based   screening rule, which at each iteration removes some of the non-support atoms of the true solution $x^*$.  This method is safe for convex penalties and a useful heuristic for nonconvex penalties; for all penalties it converges in finite time to the true support. Having this information can improve caching for improving subproblem efficiency, and can be used in two-stage methods if the method is ended early.

\item  In general, CGM without line search or away steps does not guarantee finite-time support recovery. We thus give a finite-time support identification rate of $O(1/\delta^2)$ on the post-screened atoms, where $\delta$ is a problem-dependent conditioning parameter that measures its distance to degeneracy. 
\end{itemize}

\paragraph{Outline.} We  present  the RP-CGM in three stages, with increasing complexity. In section \ref{sec:simple} we consider the nonconvex element-wise penalty, giving the key intuition behind the general method, with simple proofs and analysis. In section \ref{sec:convex} we consider the generalized convex gauge penalized problem, using P-CGM, and show how to handle simple recession cones in $\mP$. Finally, in section \ref{sec:nonconvex}, we introduce reweighting, and show that RP-CGM has an overall $O(1/t)$ convergence rate and finite support identification using our screening rule. Experimental results support this in section \ref{sec:experiments}.

\section{Reweighted Penalized CGM for simple sparse recovery}

\label{sec:simple}
We begin by considering the optimization problem 
\begin{equation}
    \minimize{x\in \R^d} \quad F(x) := f(x) + \underbrace{\phi(r(x))}_{h(x)}, \qquad r(x) = \sum_{i=1}^d \gamma(|x_i|).
    \label{eq:main-1norm}
\end{equation}
This is the simplification of \eqref{eq:main0} with $r := r_\mP$ and $\mP_0 = \{\pm e_1,...,\pm e_d\}$ the signed unit basis. 
The more general case of the $r_\mP$ gauge-like penalty follows a  similar analysis to what is presented in this section, and can be viewed intuitively as sparsity in a preimage space.
We define the \emph{support of $x$} as the indices of the nonzeros as
$\supp(x) = \{i: x_i \neq 0\}$.
For a method producing iterates $x^{(1)}, x^{(2)}, \to x^*$, we say that this method has \emph{recovered the support at iteration $\bar t$} if for all $t' \geq \bar t$, $\supp(x^{(t)}) = \supp(x^*)$.


\subsection{Stationary points}

For a continuous function $h:\R^d\to\R$, the point $x^*$ is a \emph{Clarke stationary point}   of \eqref{eq:main-1norm} if 
$0\in \nabla f(x^*) + \partial h(x^*)$ where $\partial h(x) = \conv\; \{\lim_{x'\to x} \nabla h(x')\}$
is the \emph{Clarke subdifferential} of $h$ at $x$ \citep{clarke1975generalized, clarke1983nonsmooth}.

\begin{assumption}[$\phi$ and $\gamma$]
Assume that 
\begin{itemize}
    \item $\phi:\R_+\to\R_+$ is convex, monotonically increasing, and differentiable everywhere on its domain,
    \item $\gamma:\R_+\to\R_+$ is concave, monotonically increasing, and differentiable everywhere on its domain, 
    \item the derivative $\gamma'(\xi)$ is lower and upper bounded by 
    \[
    0 < \gamma_{\min} := \lim_{\xi\to _\infty} \gamma'(\xi) \leq \gamma'(\xi) \leq  \lim_{\xi\to 0^+} \gamma'(\xi) =:\gamma_{\max} < +\infty,  \quad \forall \xi \geq 0.
    \]
\end{itemize}
\label{asspt:phigamma1}
\end{assumption}
From these assumptions,  the Clarke subdifferential for $h(x)$ is 
\[
(\partial h(x))_i=
\begin{cases}
\phi'(r(x))\cdot \{\sign(x_i) \, \gamma'(|x_i|)\}, & x_i \neq 0,\\
\phi'(r(x))\cdot [-\gamma_{\max}, \gamma_{\max}], & x_i = 0.
\end{cases}
\]
In other words, the optimality conditions can  be summarized as follows: $x^*$ is a stationary point of \eqref{eq:main-1norm} if 
\begin{eqnarray*}
x^*_i \neq 0 \; &\Rightarrow& \; -\nabla f(x^*)_i =   \phi'(r(x))  \;  \gamma'(|x_i|) \\
x^*_i = 0\; &\Rightarrow& \; -\nabla f(x^*)_i \in  \phi'(r(x))\;  [-\gamma_{\max}, \gamma_{\max}].
\end{eqnarray*}

\paragraph{Example.}
Consider the concave regularizer
$h(x) := \bigg(\sum_i \sqrt{|x_i|+ \xi_0}\bigg)^2$. 
This construction arises from $\phi(\xi) = \xi^2$ and $r(\xi) = \sqrt{|\xi|+\xi_0}$. 
Its Clarke-subdifferential can be expressed element-wise
\[
(\partial h(x))_i  = \left(\sum_j \sqrt{|x_j|+\xi_0}\right)  \cdot
\begin{cases}
[-\xi_0^{-1/2},\xi_0^{-1/2}], & x_i = 0,\\
\left\{ \frac{\sign(x_i)}{\sqrt{|x_i|+\xi_0}} \right\}, & |x_i| > 0.
\end{cases}
\]
The optimality conditions can also be summarized in terms of ``wiggle room"; that is, whenever $x_i = 0$, then $\nabla f(x)_i$ lies in an interval. But when $x_i \neq 0$, $\nabla f(x)_i$ must take a specific value.  When $\gamma$ is concave, however, that value could lie anywhere in the interval $[-\xi_0^{-1/2},\xi_0^{-1/2}]$. In this case, the gradient $\nabla f(x)$ cannot reveal any hints as to the support of $x^*$, even as $x \to x^*$.

\paragraph{Concave penalties} Popular concave penalties are listed in Table  \ref{tab:concavepenalties}. To satisfy assumption \ref{asspt:phigamma1},
we propose a piecewise extension given a concave penalty $\gamma_0$ and a ``boundary point" $\bar \xi$:
\begin{equation}
\gamma(\xi) = 
\begin{cases}
\gamma_0(\xi), & 0 \leq \xi \leq \bar\xi,\\
\gamma_0'(\bar\xi)(\xi - \bar\xi ) + \gamma_0(\bar\xi) , & \xi < \bar\xi.
\end{cases}
\label{eq:locallyconvex}
\end{equation}
This ensures that $\gamma_{\max} = \gamma_0'(\bar \xi) < +\infty$. 
A more complete table of commonly used concave penalties is given by \citet{gong2013general, rakotomamonjy2019screening}.
See also figure \ref{fig:phigamma}.

\begin{table}
\centering
\begin{tabular}{|r|l|c|c|}
\hline
& $\gamma(c)$ & $\displaystyle\lim_{c\to 0} \gamma'(c)$ & $\displaystyle\lim_{c\to +\infty} \gamma'(c)$\\\hline
Fractional fns & $q^{-1} c^q$, $0 < q < 1$ & $+\infty$ & 0\\&&&\\
LSP & $ \log(1+|c|/\theta)$ for $\theta > 0$ & $\theta^{-1}$ & 0\\&&&\\
SCAD & $\begin{cases}
\lambda |c| & |c| \leq \lambda, \\
\frac{-c^2 + 2\theta\lambda |c|-\lambda^2}{2(\theta-1)}& \lambda < |c| \leq \theta\lambda, \\
(\theta+1)\lambda^2/2& |c| \geq \theta\lambda, 
\end{cases}$ for $\theta > 2$ & $\lambda$ & 0 \\&&&\\
MCP & $\begin{cases}
\lambda |c| - c^2/(2\theta) & |c| \leq \theta \lambda,\\
\theta \lambda^2/2 & |c| > \theta\lambda,
\end{cases}$ for $\theta > 0$ & $\lambda$ & 0\\&&&\\
Locally convex & \eqref{eq:locallyconvex}, given $\gamma_0$ and $\bar \xi$ & $\displaystyle\lim_{c\to 0}\gamma_0'(c)$ & $\gamma_0'(\bar \xi)$\\\hline
\end{tabular}
\label{tab:concavepenalties}
\caption{A list of several popular concave penalties, and their slope behavior at extremities. The last entry shows the effect of the piecewise construction, which becomes linear with non-zero slope at large values of $\xi$.}
\end{table}

\subsection{Convex majorant}
Inspired by methods in majorization-minimization and difference-of-convex literature, we propose the RP-CGM, which at each iteration takes a conditional gradient step over the following convex proxy problem
\begin{equation}
    \mini{x\in \R^d} \quad \bar F(x;x^{(t)}) := f(x) + \phi(r(x^{(t)}) - \bar r(x^{(t)};x^{(t)}) + \bar r(x; x^{(t)})).
\label{eq:lin_primal_onenorm}
\end{equation}
where $\bar r(x;\bar x) := \sum_i \gamma'(|\bar x_i|) |x_i|$ is the linearized function of $x$ with reference point $\bar x$.
Expanding the terms, given $\gamma$ concave,
\begin{equation}
\sum_{i=1}^d \underbrace{\gamma(|x^{(t)}_i|) }_{r(x^{(t)}) } + \underbrace{\gamma'(|x^{(t)}_i|)(|x_i| - |x^{(t)}_i|)}_{\bar r(x; x^{(t)}) - \bar r(x^{(t)}; x^{(t)})} \geq \sum_{i=1}^d \gamma(|x_i|) = r(x)
\end{equation}
shows that $\bar F^{(t)}(x) \geq F(x)$ for all $x^{(t)}$, and is indeed a majorant; additionally, when $x = x^{(t)}$, we have $\bar F^{(t)}(x) = F(x)$.


\begin{figure}
    \centering
    \begin{tabular}[t]{ccc}
    \begin{minipage}{3in}
    \includegraphics[width=3in]{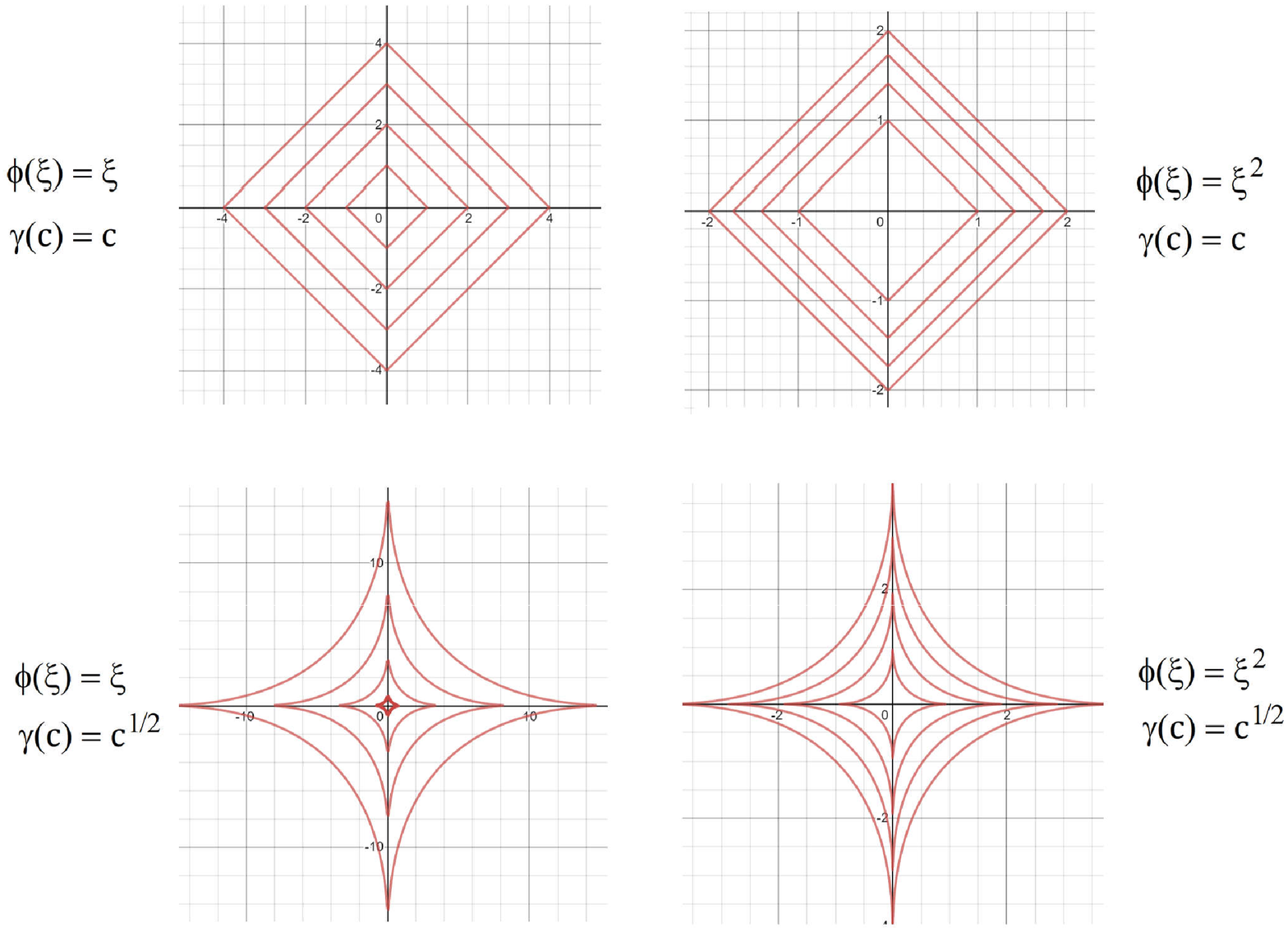}
    \end{minipage}
    & & 
    \begin{minipage}{1.15in}
     \includegraphics[width=1.25in]{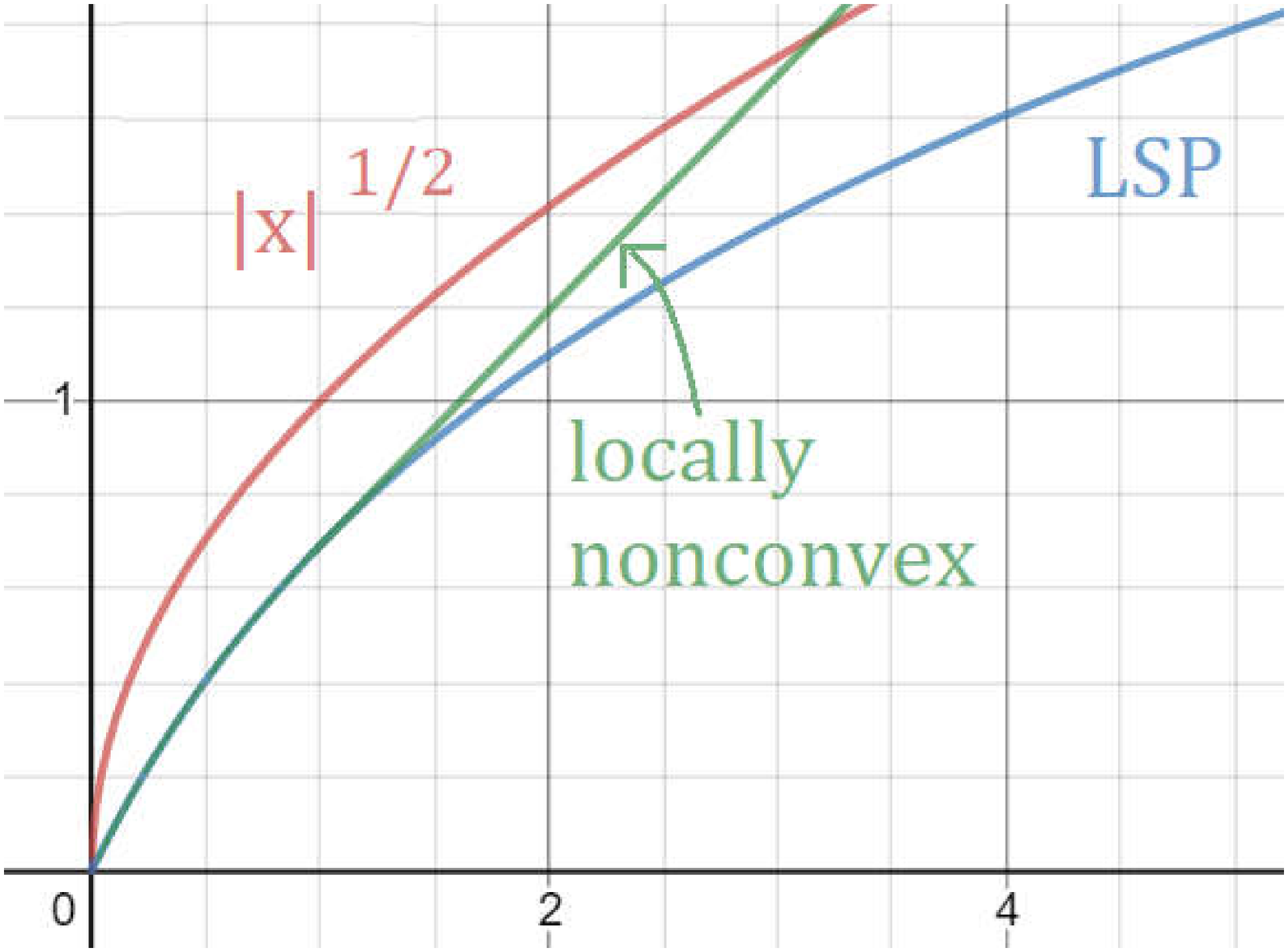}
    \end{minipage}
\end{tabular}
    \caption{\textbf{Transformations $\phi$ and $\gamma$.} Left: Level sets for the penalty $h(x) = \phi(\sum_i \gamma(|x_i|))$. The concave penalty $\gamma$ increases the ``spike-ness"; the convex penalty $\phi$ increases the effect of the aggregate value.
    Right: Three example functions of $\gamma$.
     RP-CGM will behave erratically when $\gamma_{\min} = 0$ (red and blue) and $\gamma_{\max}$ is unbounded (red), so we use a penalty that is bounded on both ends (green).
    }
    \label{fig:phigamma}
\end{figure}

\paragraph{Primal-dual pair.} Given a reference point $\bar x$, we define $r_0 := r(\bar x) - \bar r(\bar x; \bar x)$ and $w_i = \gamma'(|\bar x_i|)$. We can now ignore the reference point, given that $r_0$ and $w$ are fixed.
The penalty of the linearized problem is then 
\[
\bar h(x) = \phi(r_0+\bar r(x)), \qquad \bar r(x) = \sum_i w_i|x_i|.
\]
From assumption \ref{asspt:phigamma1}, for all $i$ $w_i \geq \gamma_{\min} > 0$. Then the convex conjugate of the penalty function is 
\begin{eqnarray*}
\bar h^*(y) &=& \sup_{x}\; x^Ty - \phi(r_0+\sum_iw_i|x_i|)\\
 &=& \sup_{\hat x, \xi}\;\{\xi \hat x^Ty - \phi(r_0+\xi) : \sum_i w_i |\hat x_i| = 1\}\\
&\overset{u_i = w_i\hat x_i, \; v_i = y_i/w_i}{=}& \sup_{u,\xi}\;\{ \xi u^Tv - \phi(r_0 + \xi) : v_i = \frac{y_i}{w_i}, \|u\|_1 = 1\}.
\end{eqnarray*}
In particular, note that $\bar r(x)$ is positive homogeneous in $x$, which is why we can decompose $x = \xi \hat x$ where $\bar r(x) = \xi$ and $\hat x = \xi^{-1} x$, optimizing the two independently. 
The maximization over $u$ does not involve $\phi$, and is achieved at $u = \sign(v_k)e_k$ where $|v_k| = \|v\|_\infty$. Then 
\begin{eqnarray*}
\bar h^*(y) &=&  \sup_{\xi}\;\{ \xi \|v\|_\infty- \phi(r_0 + \xi) + \underbrace{r_0\|v\|_\infty-r_0\|v\|_\infty}_{=0}: v_i = \frac{y_i}{w_i}\}\\
&=& \phi^*(\|v\|_\infty) - r_0 \|v\|_\infty.
\end{eqnarray*}
We can therefore   we rewrite  \eqref{eq:lin_primal_onenorm}  and its Fenchel dual as
\begin{eqnarray*}
\psimple&&
 \mini{x\in \R^d} \quad \bar F(x; \bar x) := f(x) + \underbrace{\phi\left(r_0 + \sum_i w_i |x_i|\right)}_{=:\bar h(x;\bar x)}\\
\dsimple&&
\maxi{z} \quad \bar F_D(z; \bar x) := -f^*(-z) - \phi^*\left(\max_i \frac{|z_i|}{w_i}\right) + r_0 \cdot \left(\max_i \frac{|z_i|}{w_i}\right),
\end{eqnarray*}
where if $\bar x^*$ optimizes \psimple~then $-\nabla f(\bar x^*)$ optimizes \dsimple.

\paragraph{Optimality conditions.} $\bar x^*$ minimizes \psimple~if and only if
\[
0\in \nabla f(\bar x^*) + \partial \bar h(\bar x^*;\bar x).
\]

\begin{assumption}
We assume that $f$ is convex and $L$-smooth w.r.t. $\|\cdot\|_1$:
\begin{equation}
f(y) - f(x) \leq \nabla f(x)^T(y-x) + \frac{L}{2}\|y-x\|_1^2,\quad \forall x,y.
\label{eq:lsmooth-l1}
\end{equation}
\label{asspt:Lsmooth-l1}
\end{assumption}
An important consequence of \eqref{eq:lsmooth-l1} is that, while the set of minimizers of \psimple~may not necessarily be unique, their gradient $\nabla f(\bar x^*)$ will be unique. Specifically, 
\eqref{eq:lsmooth-l1} implies that 
\begin{equation}
f(x)-f(y) \geq \nabla f(y)^T(x-y) +  \frac{1}{2L}\|\nabla f(x)-\nabla f(y)\|_\infty^2, \quad \forall x,y
\label{eq:stronglyconvex:1norm}
\end{equation}
and in particular taking $y = \bar x^*$ where $-\nabla f(\bar x^*) \in \partial \bar h(\bar x^*;\bar x)$, we have 
\[
 f(x)+ \bar h(x;\bar x)-f(\bar x^*) - \bar h(\bar x^*;\bar x) \geq  \frac{1}{2L}\|\nabla f(x)-\nabla f(\bar x^*)\|_\infty^2, \quad \forall x
\]
and thus $x$ is optimal only if $\nabla f(x) = \nabla f(\bar x^*)$.

\subsection{Reweighted penalized CGM}
The RP-CGM on \eqref{eq:main-1norm}  solves at each iteration the \minmaj~step
\begin{equation}
s = \argmin{s\in \R^d} \; \nabla f(x)^Ts + \phi\left(r_0 + \sum_i w_i |s_i|\right).
\label{eq:lmo-l1}
\end{equation}
 Following the derivation of the penalty conjugate, the optimization can be written as a composition of two steps, using a decomposition 
$s = \xi \hat s$, $\xi = \sum_i w_i |s_i|$,
 and noting that  $\sum_i |w_i| |s_i|$ is positive homogeneous w.r.t. $s$, the two variables can be treated separately. 
 Specifically, 
 \[
\hat s = \frac{\sign(v_k)}{|w_k|}e_k, \quad k = \argmax{k}\; |v_k| , \quad v_i = \begin{cases}
-\nabla f(x)_i/|w_i|, & w_i \neq 0,\\
0, & w_i = 0,
\end{cases}
\]
and
\[
\xi = \argmin{\xi \geq 0}\; -\xi \cdot \|v\|_\infty + \phi(r_0 + \xi).
\]


This second step does not appear in the usual CGM, but is required for dealing with sparse penalties rather than constraints. It involves a one-dimensional optimization problem that at worst can be solved efficiently through bisection; more commonly, explicit low-cost expressions exist.
In particular, 
\begin{eqnarray*}
\phi^*(\|v\|_\infty) := \sup_{\xi} \; \xi\cdot \|v\|_\infty - \phi(\xi)
=\sup_{r_0 + \xi} \; (\xi+r_0)\cdot \|v\|_\infty - \phi(r_0 + \xi),
\end{eqnarray*}
and in fact 
\[
(\phi^*)'(\|v\|_\infty) = \argsup{r_0 + \xi'} \; (\xi'+r_0)\cdot \|v\|_\infty - \phi(r_0 + \xi') = \xi + r_0\cdot \|v\|_\infty.
\]
Often, closed-form solutions exist for this mapping $(\phi^*)'$.
To relate to the vanilla CGM, where $\phi(\xi) = \iota_{\cdot \leq 1}(\xi)$, the convex conjugate $\phi^*(\nu) = \nu$ and  is always optimized at $\xi = 1$.

\paragraph{Slope requirement on $\gamma$.}
 To see why the bounds $0 < \gamma_{\min}<\gamma_{\max}<+\infty$ are important, suppose that  at any point $\gamma'(|x_i|)= 0$. Then $v_i = \|v\|_\infty \to +\infty$, and $\xi = (\phi^*)'(\|v\|_\infty) \to +\infty$. Conversely if $\gamma'(|x_i|) = +\infty$ then $u_i = +\infty$. In either case, the bounds on the slope of $\gamma$ are required to control the norm of the iterate $x$. (See also fig. \ref{fig:phigamma}.)

\paragraph{Minimum curvature on $\phi$.}
A key issue with changing the constrained formulation used in CGM to a penalized form is that without a norm bound on the iterates, it is not clear if the method will diverge. As an example, consider the minimization of $f(x) + g(x)$ where $f(x) = \tfrac{1}{2} x^2$, $g(x) = |x|$. 
Then for any $|x| > 1$, the \minmaj~step 
$\min_s   s\cdot x - |s| = -\infty$ is undefined.
We therefore require a sufficient amount of curvature in $\phi$ for convergent iterates.

\begin{assumption}[Lower quadratic bound]
We assume $\phi$ is lower-bounded by a quadratic function
$\phi(\xi) \geq \mu_\phi \xi^2-\phi_0$,
for some $\mu_\phi > 0$ and  $\phi_0$. 
\label{asspt:phi2}
\end{assumption}
\begin{property}
If \eqref{asspt:phi2} holds, then
the derivative of $\phi^*$ is asymptotically nonexpansive; e.g., for some finite-valued $\xi_0$,
$(\phi^*)'(\nu) \leq \frac{\nu}{\mu_\phi}+\xi_0$.
\end{property}
\begin{proof}
Assume that $\phi_0$ is as large as possible; e.g., there exists some finite $\xi_0$ where $\phi(\xi_0) = \mu\xi_0^2 - \phi_0$. Then for all $\xi \geq \xi_0$, for all $\nu\in \partial \phi(\xi)$,
\[
\mu (\xi^2-\xi_0^2 )\leq \phi(\xi)-\phi(\xi_0) \leq \nu(\xi-\xi_0),
\]
and therefore
\[
\nu\geq\mu \frac{(\xi+\xi_0)(\xi-\xi_0)}{\xi-\xi_0}  = \mu \xi+\mu \xi_0 \; \iff \; \xi \leq  \mu^{-1} \nu -  \xi_0.
\]
Thus, for any $\xi$,  $\nu\in \partial \phi( \xi)$ must satisfy 
$\xi \leq \max\{\xi_0,  \mu^{-1} \nu -  \xi_0\} \leq   \mu^{-1} \nu +\xi_0$.
By Fenchel Young, this must apply to all $\xi \in \partial \phi^*(\nu)$.
\qed
\end{proof}

\paragraph{Example: Monomials.}
For $1 \leq \alpha,\beta \leq +\infty$, the following $\phi:\R_+\to\R_+$ and $\phi^*:\R_+\to\R_+$ form a conjugate pair:
\[
\phi(\xi) = \frac{1}{\alpha} \xi^{\alpha}, \qquad \phi^*(\nu) = \frac{1}{\beta}\nu^\beta, \qquad \frac{1}{\alpha} + \frac{1}{\beta} = 1.
\]
In particular, in the case that $\alpha = 1$, then $\beta \to +\infty$, and the function 
\[
\phi^*(\nu) = \lim_{\beta\to+\infty} \frac{1}{\beta}\nu^\beta = 
\begin{cases}
0, & \nu \leq 1\\
+\infty, & \nu > 1.
\end{cases}
\]
In this case, whenever $\nu > 1$ then $\phi^*(\nu) = +\infty$; we exclude this case as P-CGM will not converge in this case. 
When $\alpha \geq 2$, $\phi$ is strongly convex and we can show $O(1/t)$ convergence of P-CGM.
When  $1 < \alpha < 2$, $\phi^*(\nu)$ is finite and the iterates are well-defined, but  the method may converge or diverge.


 \paragraph{Example: Barrier functions.}
Consider 
\begin{equation}
\phi(\xi) = -\frac{1}{\beta}\log(C-\xi) - \frac{\xi}{C\beta} + \frac{\log(C)}{\beta},
\label{eq:logbarrier}
\end{equation}
which is a log-barrier penalization function for $\xi \leq C$; as $\beta \to +\infty$, $\phi(\xi)$  approaches the indicator function for this constraint. Its conjugate is 
\[
\phi^*(\nu) = C\nu   - \beta^{-1} \log(C\beta\nu+1),
\]
achieved at $\xi =C^2\beta\nu/(C\beta \nu+1)$. For all $C > 0, \beta > 0$, and $\nu\neq -(C\beta)^{-1}$, both $\phi^*$ and $\xi^*$ exist and are finite. Note also the implicit constraint, as $\phi(\kappa_\mP(x))$ is finite only if $x\in C\mP$.

The entire RP-CGM algorithm for the simple sparse optimization problem is given in Alg. \ref{alg:simple}

\begin{algorithm}
\caption{RP-CGM on simple sparse optimization}
\begin{algorithmic}[1]

\Procedure{RP-CGM}{$f$, $\phi$, $\gamma$,  $T$}       
    \State Initialize with any $x^{(0)}\in \R^d$
    \For{$t = 1,...T$}
        \State Compute negative gradient $z = -\nabla f(x^{(t)})$

        \State Compute weights $w$ and offset $r_0$ \Comment{Reweight}
        \[
        w_i = \gamma'(|x^{(t)}_i|), \quad i = 1,...,d, \qquad r_0 = r(x^{(t)}) - \bar r(x^{(t)};x^{(t)})
        \]
        \State Compute next atom $s =\xi \cdot \sign(v_k) e_k $ where $v_i = z_i / w_i$, and\Comment{Min-maj}
        \[
        k = \argmax{i}\; |v_i|, \quad \xi = (\phi^*)'(\|v\|_\infty)-r_0
        \]
        
        \State Update $x^{(t+1)}$ \Comment{Merge}
        \[
        x^{(t+1)} = (1-\theta^{(t)})x^{(t)} + \theta^{(t)} s, \quad \theta^{(t)} = 2/(1+t)
        \]
    \EndFor
\Return  $x^{(T)}$
\EndProcedure
\end{algorithmic}
\label{alg:simple}
\end{algorithm}

\subsection{Convergence of RP-CGM}
The duality gap of the nonconvex original problem \eqref{eq:main-1norm}  is (as expected) bounded away from 0, and is thus an inadequate measure of suboptimality. 
\begin{property}[Duality gap of nonconvex regularizer] 
For $h(x) = \phi(r(x))$,
\begin{eqnarray*}
h(x) - h^{**}(x) \geq \phi(r(x)) - \phi(\gamma_{\min} \|x\|_1).
\end{eqnarray*}
\end{property}

\begin{proof}
First,
\[
h^*(z) = \sup_x \; x^Tz - \phi(r(x))
=  \sup_x \phi^*\left(\frac{x^Tz}{r(x)}\right).
\]
Picking 
\[
\begin{cases}
x_i\to \sign(z_i) \cdot \infty &\text{ if }  |z_i| = \|z\|_\infty\\
x_i=0 & \text{ otherwise } 
\end{cases}
\quad \Rightarrow \quad 
\frac{x^Tz}{r(x)} = \frac{\|z\|_\infty}{\gamma_{\min}}
\]
gives
$h^*(z) \geq \phi^*\left(\frac{\|z\|_\infty}{\gamma_{\min}}\right)$, and so  
\[
h^{**}(x)\leq  \sup_z \; x^Tz -  \phi^*\left(\frac{\|z\|_\infty}{\gamma_{\min}}\right)
= \phi(\|x\|_1 \gamma_{\min}).
\]
\qed
\end{proof}
Instead, we measure convergence via the gap of the linearized problem at $\bar x = x$.
\begin{property}[Residual]
The gap of the linearized problem at $x$ with reference point $x$, denoted as  $\res(x) :=\gap(x,-\nabla f(x);x)$, satisfies
\[
 \res(x)  \geq 0\;\forall x, \qquad \res(x) =0 \iff \text{ $x$ is a stationary point of 
 (\ref{eq:main-1norm}).}
\]
\label{prop:gapres-l1}
\end{property}

\begin{proof}
Since $\res(x)$ is a duality gap, it is always nonnegative. Explicitly,
denote $\nu = \max_i \left(\frac{|\nabla f(x)_i|}{\gamma'(|x_i|)}\right)$. 
 Then, since $f(x) + f^*(\nabla f(x)) = \nabla f(x)^Tx$,
\begin{eqnarray*}
\res(x) 
&=& x^T\nabla f(x) + \phi( r(x)) +  \phi^*(\nu)+ (\bar r(x;x)- r(x))\cdot \nu\\
&\overset{(a)}{\geq}& x^T\nabla f(x) +  r(x)\nu+ (\bar r(x;x)- r(x))\cdot \nu\\
&= & x^T\nabla f(x) +\sum_i \gamma'(|x_i|)|x_i| \cdot \max_j \left(\frac{|\nabla f(x))_j|}{\gamma'(|x_j|)}\right)\\
&\overset{(b)}{\geq} & x^T\nabla f(x) -x^T\nabla f(x) = 0.
\end{eqnarray*}
Tightness of (a) occurs iff Fenchel-Young is satisfied with equality, e.g. $\phi'(r(x)) = \nu$.
Tightness of (b) occurs iff 
\begin{equation}
\max_j \;\frac{|\nabla f(x)_j|}{\gamma'(|x_j|)} = \frac{-\nabla f(x)_i \cdot \sign(x_i)}{\gamma'(|x_i|)}, \quad \forall x_i \neq 0.
\label{eq:optcond-helper2}
\end{equation}
Combining these two observations, then $\res(x) = 0$ if and only if
\[
-\nabla f(x)_i \in \phi'(r(x)) \cdot \begin{cases}
\{\gamma'(|x_i|) \sign(x_i)\} & x_i \neq 0\\
 [-\gamma'(0),\gamma'(0)] & x_i = 0,
\end{cases}
\]
which is the condition for $x = x^*$ a stationary point of \eqref{eq:main-1norm}.
\qed
\end{proof}

\begin{thm}[Convergence of RP-CGM, simple case]
Pick any $x^{(0)}\in \R^d$ where $h(x^{(0)})$ is finite. Define the sequence $x^{(t)}$, $t = 1,...$ by the steps dictated in  \minmaj~and \merge, using the step size sequence $\theta^{(t)} = 2/(1+t)$. Given assumptions \ref{asspt:phigamma1}, \ref{asspt:Lsmooth-l1}, \ref{asspt:phi2}, 
then 
\[
F(x^{(t)}) - F(x^*) = O(1/t), \qquad \min_{t' \leq t}\,\res(x^{(t')}) = O(1/t).
\]
\label{th:convergence-l1}
\end{thm}
This is a special case of Theorem \ref{th:convergence}, which is proven in section \ref{sec:nonconvex} and appendix \ref{app:convergence}. The proof is inductive, and shows that $O(1/t)$ behavior ``kicks in" at a large enough $t$; explicit constants are given in section \ref{sec:nonconvex}.

\subsection{Convex support recovery and screening}
From optimality conditions,
$\bar x^*$ minimizes \psimple~if 
\begin{equation}
- \frac{\nabla f(\bar x^*)_i}{\alpha w_i} \in 
\begin{cases}
\{\sign(\bar x^*_i)\}, & \bar x^*_i \neq 0\\
[-1,1], &\bar x^*_i = 0,
\end{cases}
\label{eq:main-lin-onenorm-opt}
\end{equation}
for some $\alpha \in \partial_{\xi} \phi(r_0 + \xi)$ at $\xi = \sum_i w_i\bar x_i^*$.
In other words, for this convex reweighting problem, the sparsity pattern of $\bar x^*$ can be partially ascertained from $\nabla f(\bar x^*)$, in that the set of nonzeros of $\bar x^*$ must be contained in the set of maximal indices of the reweighted $\nabla f(\bar x^*)$. Formally, define 
\begin{equation}
\dsupp(x; \bar x) :=\left\{i : \frac{|\nabla f(x)_i|}{\gamma'(|\bar x_i|)} =\max_j \frac{|\nabla f(x)_j|}{\gamma'(|\bar x_j|)}  \right\}.
\label{eq:dsupp-l1}
\end{equation}
Then the optimality condition \eqref{eq:main-lin-onenorm-opt} states that $\supp(\bar x^*) \subseteq \dsupp(\bar x^*;\bar x)$, where $\bar x^*$ minimizes \psimple. We are in particular interested in $\bar x^* = x^*$ the stationary point of \eqref{eq:main-1norm}.
From this observation, we have our first screening property
\begin{property}[Screening for simple sparsity.]
If $\|\nabla f(x)-\nabla f( x^*)\|_\infty \leq \epsilon$, then 
\begin{equation}
\|\nabla f(x)\|_\infty - |\nabla f(x)_i| >  2\epsilon\gamma_{\max}\; \Rightarrow \;  x_i^* = 0.
\label{eq:screen-l1-prop}
\end{equation}
\end{property}
\begin{proof}
First, define $v_i = \frac{|\nabla f(x)_i|}{w_i}$ and $\bar v_i^* = \frac{|\nabla f( x^*)_i|}{w_i} $. 
Then 
\[
\frac{\|\nabla f(x^*)-\nabla f(x)\|_\infty }{\gamma_{\max}}\leq  \frac{\|\nabla f(x^*)-\nabla f(x)\|_\infty }{\max_i w_i}  \leq \|v-\bar v^*\|_\infty.
\]
By optimality conditions 
$\bar v_i^* < \|\bar v^*\|_\infty \Rightarrow \bar x_i^* = 0$.
Thus, \eqref{eq:screen-l1-prop} implies
\[
2\epsilon >   \frac{\|\nabla f(x)\|_\infty - |\nabla f(x)_i|)}{\gamma_{\max}} \geq 
\|\bar v\|_\infty-\bar v_i
\]
and therefore
$\|\bar v^*\|_\infty-\bar v^*_i \leq \|\bar v\|_\infty-\bar v_i + 2\epsilon  < 0$.
\qed
\end{proof}


While this fundamental concept is very intuitive, it is not practical, as one does not have access to $\nabla f(x^*)$ until the algorithm has terminated. 
In the convex setting, this is overcome by showing that the duality gap between \psimple~and \dsimple~upper bounds the gradient error. In the nonconvex setting, this bound is more complicated, but still leads to a useful heuristic screening rule.
\begin{property}[Residual bound on gradient error.]
Define $D(x) = \sum_{i=1}^d \gamma(|x_i|)-\gamma(|x_i^*|) - \gamma'(|x_i|)(|x_i-x_i^*|)$
the linearization error at $x$. 
Denoting $x^*$ a stationary point of \eqref{eq:main-1norm}, then
\[
\|\nabla f(x^*)-\nabla f(x)\|_\infty \leq \frac{LD(x)}{2\gamma_{\min}} + \sqrt{\frac{L^2D(x)^2}{4\gamma^2_{\min}} + L\res(x) + LD(x) \frac{\|\nabla f(x)\|_\infty}{\gamma_{\min}}}.
\]
\label{prop:resboundsgrad-l1}
\end{property}
Note that if $r(x) = \|x\|_1$ then $D(x) = 0$.
\begin{proof}
First, note that 
\begin{eqnarray}
\phi^*\left(\max_i \frac{|z_i|}{w_i}\right) + r_0\cdot \left(\max_i \frac{|z_i|}{w_i}\right)  &=&  \sup_y\; y^Tz - \phi(r_0 + \sum_i w_i |y_i|) \nonumber\\
&\geq &   z^Tx^* - \phi(r_0 + \sum_i w_i |x^*_i|). \label{eq:phistar-ell1-proofhelper1}
\end{eqnarray}
As before, $\res(x) = \bar F(x;x)-\bar F_D(-\nabla f(x);x)$.
Taking $(x,-\nabla f(x))$ as a feasible primal-dual pair and reference point $\bar x=x$, and denoting $w_i =\gamma'(|\bar x_i|)$, $\epsilon(x) = \phi(r_0+\sum_i w_i |x_i^*|) - \phi(r(x^*))$, and $\nu = \max_i \frac{|\nabla f(x)_i|}{w_i}$,
\begin{eqnarray*}
\res( x) &=&\underbrace{f(x) + f^*(\nabla f(x))}_{\text{use Fenchel-Young}}+\phi(r(x)) + \underbrace{\phi^*(\nu) -r_0 \cdot (\nu) }_{\text{use \eqref{eq:phistar-ell1-proofhelper1}}} \\
&\geq & \nabla f(x)^T(x- x^*) + \phi(r(x))   - \phi\left(r_0 + \sum_i w_i | x^*_i|\right)\\
&\overset{+\epsilon(x)-\epsilon(x)}{\geq}  & \nabla f(x)^T(x- x^*) + \underbrace{\phi(r(x)) -\phi(r(x^*))}_{\text{convex in $x$}}- \epsilon(x)\\
&\overset{g\in \partial h(x^*)}{\geq} & \nabla f(x)^T(x- x^*) + g^T(x- x^*) - \epsilon(x).
\end{eqnarray*}
Picking in particular $g = -\nabla f( x^*)$,
\[
\res(x) +\epsilon(x) \geq  (x- x^*)^T( \nabla f(x) -\nabla f( x^*)) 
\overset{(\star)}{\geq} \frac{1}{L}\|\nabla f( x^*)-\nabla f(x)\|_\infty^2,
\]
where $(\star)$ follows from assumption \ref{asspt:Lsmooth-l1}.


Next, note that
\[
\epsilon(x) = \phi(r(x)-r(x;x)+r(x^*;x))-\phi(r(x^*)) 
\leq \phi'(r(x^*))D(x),
\]
where in general,  $D(x) \leq (\gamma_{\max}-\gamma_{\min})\|x-x^*\|_1$ and $D(x) = 0$ if $\gamma(\xi) = \xi$ (convex case).
Noting that, at optimality, 
\[
\phi'(r(x^*)) = \max_i\frac{|\nabla f(x^*)_i|}{\gamma'(|x_i^*|)}\leq \frac{\|\nabla f(x^*)\|_\infty}{\gamma_{\min}},
\]
then
\[
\gamma_{\min}\phi'(r(x^*)) \leq \|\nabla f(x^*)\|_\infty \leq  \|\nabla f(x)\|_\infty + \|\nabla f(x) - \nabla f(x^*)\|_\infty 
\]
and overall,
\begin{eqnarray*}
\|\nabla f(x^*)-\nabla f(x)\|_\infty^2 
&\leq& L \res(x) + LD(x)\frac{  \|\nabla f(x)\|_\infty + \|\nabla f(x) - \nabla f(x^*)\|_\infty }{\gamma_{\min}}.
\end{eqnarray*}
This inquality is quadratic in $\|\nabla f(x)-\nabla f(x^*)\|_\infty$, which leads to the stated bound. 
\qed
\end{proof}

From these two properties, we immediately get a screening rule for \eqref{eq:main-1norm}:
\begin{thm} [Heuristic screening rule.]
 For any $x$ and some choice of $\epsilon > 0$, define 
\begin{equation}
\mI_\epsilon^{(t)} = \left\{i : \|\nabla f(x)\|_\infty - |\nabla f(x)_i| \geq    \epsilon + 2\sqrt{L\res(x)+\epsilon } \right\}.
\label{eq:screening-l1}
\end{equation}
If 
\[
\epsilon \geq \frac{L D(x)}{\gamma_{\min}} \; \max\left\{\frac{1}{2},\frac{L D(x)}{4\gamma_{\min}} + \|\nabla f(x)\|_\infty\right\}
\]
then 
 $ x_i^* = 0$ for all $i\in \mI_\epsilon^{(t)}$, where  $ x^*$ any minimizer of \eqref{eq:main-1norm}. 
\label{th:screen-l1}
\end{thm}
Note that in the convex case ($\gamma(|x_i|) = |x_i|$) then $D(x) = 0$ and $\epsilon = 0$ is a safe choice, for all $x$. In the general case, since we do not know $D(x)$, we cannot guarantee the safety of an intermediate iterate; however, since $D(x^*) = 0$ by definition of stationary point, then $x^{(t)}\to x^*$ implies $D(x^{(t)}) \to 0$. Picking any decaying sequence $\epsilon^{(t)}\to 0$, therefore, forms a rule that  converges to the true support.

\paragraph{Degeneracy and support recovery guarantee.}
Following the terminology introduced in \cite{hare2011identifying}, we say that $x^*$ is a \emph{degenerate solution} if $\supp(x^*) \neq \dsupp(x^*;x^*)$; that is, there exists $i$ where 
\[
x_i^* = 0 \quad \text{ and }  \quad \frac{|\nabla f(x^*)_i|}{\gamma'(|x^*_i|)} = \max_j\, \frac{|\nabla f(x^*)_j|}{\gamma'(|x^*_j|)}.
\]
Finite-time support recovery is not possible for degenerate solutions \citep{lewis2011identifying,hare2011identifying,burke1988identification}. We define
\[
\delta_i(x) = \max_j\, \frac{|\nabla f(x)_j|}{\gamma'(|x_j|)} - \frac{|\nabla f(x)_i|}{\gamma'(|x_i|)},\quad \delta_{\min}(x) = \min_{i : x_i = 0} \delta_i(x)
\]
and
the quantity $\delta_{\min}(x^*)$ expresses the distance to degeneracy for this solution.  

\begin{corollary}
If $\delta_{\min} > 0$, then for a method $x^{(t)}\to x^*$,
then the  screening rule \eqref{eq:screening-l1} with $\epsilon = 0$
identifies $\supp(x^*)$ after a finite   number of iterations $\bar t$; that is, for all $t \geq  \bar t$, $\mI^{(t)}_0 = \supp(x^*)$. In the convex case ($\gamma(|x_i|) = |x_i|)$, this occurs when $\|\nabla f(x^*)-\nabla f(x)\|_\infty\leq \delta_{\min}/3 $, which occurs at $\bar t = O(1/\delta_{\min}^2)$.
\end{corollary}


\section{P-CGM for general convex sparse optimization}
\label{sec:convex}

We now extend our notion of sparsity to more generalized cases through the use of gauge functions. We begin by considering the convex restriction of \eqref{eq:main0}, which will illuminate many key properties:
\begin{equation}
\minimize{x\in\R^d} \; f(x) + \phi(\kappa_\mP(x)),
\label{eq:main-convex}
\end{equation}
where $\kappa_\mP(x)$ is the \emph{gauge} defined by a set $\mP$ at point $x$. This function generalizes the 1-norm to more size-measuring functions that include norms, semi-norms, and convex cone restrictions. In particular, when $\mP$ is the convex hull of a set of atoms, $\kappa_\mP(x)$ can be used to promote  sparsity with respect to those atoms.

In this section we focus on these convex guages, which have many useful properties for CGM and screening. In section \ref{sec:nonconvex}, we will expand the discussion to include the concave transformation that makes the sparsification more aggressive.

\subsection{Generalized support recovery}
\paragraph{Gauges and support function.} Consider a set of atoms $\mP_0\subset \R^d$. 
The function 
\begin{equation}
\kappa_\mP(x) := \min_{c_p\geq 0}\left\{\sum_{p\in \mP_0} c_p : \sum_{p\in \mP_0} c_pp = x\right\}
\label{eq:gauge}
\end{equation}
is the \emph{gauge function} over $\mP$ the closed convex hull of $\mP_0$ \citep{freund1987dual,chandrasekaran2012convex}. 
A ``dual gauge'' can be constructed as the support function 
$\sigma_\mP(z) := \sup_{s\in \mP_0} z^Ts$.
If $\kappa_\mP$ is a norm, then $\sigma_\mP$ is the usual dual norm. \citep{rockafellar1970convex,borweinlewis} .
At each iteration, a key part of implementing the variants of CGM is in finding quick ways of solving for $\sigma_\mP(z)$. 

\paragraph{Classical gauge definition.}
In convex analysis literature \citep{rockafellar1970convex, borweinlewis}, the gauge function over a closed convex set $\mP$ is defined as 
\begin{equation}
\kappa_\mP(x) := \inf\{\mu \geq 0 : x\in \mu \mP\}
\label{eq:gauge-classical}
\end{equation}
which is equivalent to \eqref{eq:gauge}.
The gauge ``contains" the diameter of the set $\mP$; e.g.
$\diam(\mP)  \kappa_\mP(x) \geq \|x\|_2$ where $\diam(\mP) := \sup_{x\in \mP,y\in \mP}\|x-y\|_2$.

\paragraph{Support recovery.}
Given a solution to \eqref{eq:gauge},  define the
 \emph{decomposition of $x$ with respect to $\mP_0$}  as tuples $c_p,p$, extracted  via the mapping $\coeff_\mP(x,p) = c_p$.
The  \emph{support of $x$ with respect to $\mP_0$} is 
\begin{equation}
\supp_\mP(x) = \{p : \coeff_\mP(x,p) > 0 \text{ in }  \eqref{eq:gauge}\}.
\label{eq:convex-gauge-def}
\end{equation}
For general $\mP$, neither the decomposition nor the support of $x$ is unique. As before, we say that  we say the support recovery is achieved if one such support $\supp_\mP(x^*)$ of the limiting point $x^{(t)}\to x^*\in \mX^*$ is revealed.
The reduction to the support definition in the previous section occurs when $\mP_0 = \{\pm e_1,...,\pm e_d\}$ the signed standard basis, Then $\supp_\mP(x)$ is unique, and explicitly
$\supp_\mP(x) = \{\sign(x_i) e_i : x_i \neq 0\}$.


\begin{property}[Support optimality condition]
\label{prop:support_opt}
Consider  the general convex sparse optimization problem \eqref{eq:main-convex}
where $\phi:\R_+\to\R_+$ is a  monotonically nondecreasing function. Then for any $x^*$ a minimizer of \eqref{eq:main-convex},
\begin{equation}
-\nabla f(x^*)^Tx^* = \kappa_\mP(x^*)\sigma_\mP(-\nabla f(x^*)).
\label{eq:alignment}
\end{equation}
and
\begin{equation}
p\in \supp(x^*) \Rightarrow 
-\nabla f(x^*)^Tp = \sigma_\mP(-\nabla f(x^*)).
\label{eq:supportrecovery}
\end{equation}
\end{property}

\begin{proof}
Without loss of generality, we assume $0\in \mP$, since $\kappa_\mP = \kappa_{\mP \cup \{0\}}$. Denote  $z^* = -\nabla f(x^*)$. 
Then the optimality condition for \eqref{eq:main-convex} is 
\begin{equation}
z^* \in \partial h(x^*)\overset{(\star)}{=} \alpha\partial \kappa_\mP(x^*), \quad h(x) := \phi(\kappa_\mP(x))
\label{eq:opthelp1}
\end{equation}
for some $\alpha \in \partial \phi(\xi)$ at $\xi = \kappa_\mP(x^*)$. 
Here, ($\star$) is a result from \citet{bauschke2011convex}, Corollary 16.72.

Since $\phi$ is monotonically nondecreasing over $\R^+$, $\alpha \geq 0$. 
If $\alpha = 0$, then $\nabla f(x^*) = 0$ and both results are trivially true.  Now consider $\alpha > 0$.
Noting that $\kappa_\mP = \sigma_{\mP^\circ}$ where $\mP^\circ$ is the polar set of $\mP$, 
\[
\alpha^{-1}z^* = \argmax{z\in \mP^\circ}\;(x^*)^Tz 
\iff (z^*)^Tx^* 
= \kappa_{\mP^\circ}(z^*)\sigma_{\mP^\circ}(x^*) =  \kappa_\mP(x^*)\sigma_\mP(z^*)
\]
which proves \eqref{eq:alignment}.
 Now take the conic decomposition $x^* = \sum_{p\in \mP_0} c_p p$ where $c_p \geq 0$, and 
\[
 (x^*)^Tz^*  = \sum_{p\in \mP_0} c_p p^Tz^* \leq \underbrace{\left(\sum_{p\in \mP_0} c_p \right)}_{=\kappa_\mP(x^*)}\underbrace{( p^Tz^*)}_{\leq \sigma_\mP(z^*)},
\]
which is with equality if and only if $p^Tz^*= \sigma_\mP(z^*)$ whenever $c_p > 0$, proving \eqref{eq:supportrecovery}. 
\qed
\end{proof}

Note the close relationship between this property and its analogous version in the previous section \eqref{eq:main-lin-onenorm-opt}. 
As before, the main idea behind this property is that a ``nonzero atom" in $x$ corresponds to a ``maximal atom" in its negative gradient $-\nabla f(x)$. We now generalize the definition of dual support from \eqref{eq:dsupp-l1}:
\[
\dsupp_\mP(x) := \{p : -\nabla f(x)^Tp = \sigma_\mP(-\nabla f(x))\}, 
\]
and Property \ref{prop:support_opt} says that for any $x$, $\supp_\mP(x) \subseteq \dsupp_\mP(x)$.

\paragraph{Distance to degeneracy.}
If for all supports $\supp_\mP(x^*) \neq \dsupp_\mP(x^*)$, then we say \eqref{eq:main-convex} is \emph{degenerate}; specifically, we cannot guarantee finite-time support recovery. Moreover, the distance to degeneracy empirically plays a role in convergence behavior; ``less degenerate" problems often converge faster and identify the support more quickly. Following the notation in \cite{nutini2019active, sun2019we}, we express this distance as $\delta_{\min}(x^*)$, where 
\[
\delta_p(x) = \sigma(-\nabla f(x)) + p^T\nabla f(x), \qquad \delta_{\min}(x) = \min_{p\in \mP}\{\delta_p(x) : p\not\in \supp_\mP(x)\}
\]
for any support of $x$. As before, $\delta_{\min}(x^*) = 0$ means the problem is degenerate.

\paragraph{Linear maximization oracle.}
The extraction of the ``maximal atoms" of a negative gradient is done through the linear maximization oracle (LMO), defined as the vector-to-vector mapping 
$\LMO_\mP(z) := \argmax{s\in \mP} \; s^Tz$.
For compact $\mP$, the set of maximizers is always nonempty and finite-valued; if the set contains more than one element,  $\LMO_\mP(z)$ returns any element in the set.
A key feature of the CGM is that this LMO is often  cheap to compute in practice, and despite weaker convergence guarantees compared to higher order methods, often converges quickly when $x^*$ is sparse with respect to structured $\mP_0$. (See also Table \ref{tab:gauges}.)

\subsection{Examples}
\paragraph{ $\ell_1$ norm.}  Consider the problem  of minimizing $f(x) + \frac{1}{2}\|x\|_1^2$.
In this case, $\sigma_\mP = \|\cdot\|_\infty$ is the dual norm of $\|\cdot\|_1$.
Then, by setting the optimality condition $0\in \partial g(x^*)$ and decomposing by index, at optimality
\[
\begin{cases}
(-\nabla f(x^*))_i = \|x^*\|_1  \,\sign(x^*_i) &\text{ if } x_i^* \neq 0, \\
(-\nabla f(x^*))_i  \in \|x^*\|_1\, [-1,1] &\text{ if } x_i^* = 0.
\end{cases}
\]
In words, the gradient of $f$ along a coordinate for which the optimal variable is nonsmooth with respect to $\kappa_\mP$ is allowed  ``wiggle room''; in contrast, if $g(x)$ is smooth in the direction of $x_i$ then the gradient is fixed. In terms of support recovery,  $\max_i |\nabla f(x^*)_i| = \|x^*\|_1$ and additionally, if ${|\nabla f(x^*)_i|  < \|x^*\|_1}$ then it must be that $x_i^* = 0$.

\paragraph{Weighted $\ell_1$ norm.}  The convex majorant in section \ref{sec:simple} specifically considered 
$\kappa_\mP(x) = \sum_i w_i |x_i|$, for weights $w_i > 0$. Here, $\mP_0 = \{\pm w_1^{-1}e_1,  \cdots \pm w_d^{-1}e_d\}$,
with corresponding ``dual gauge" $\sigma_{\mP}(z) = \max_i \frac{|z_i|}{|w_i|}$,
and the LMO follows exactly the steps for the bounded maximization computation in \eqref{eq:lmo-l1}. 
Note also that the optimality conditions of \eqref{eq:main-convex} for this choice of $\kappa_\mP(x)$ is  
\[
\begin{cases}
 \displaystyle \frac{|\nabla f(x^*)_i|}{|w_i|} \,=\, \max_j \left(\frac{|\nabla f(x^*)_j|}{|w_j|} \right)  \cdot\sign(x^*_i) &\text{ if } x_i^* \neq 0, \\[3ex]
 \displaystyle \frac{|\nabla f(x^*)_i|}{|w_i|}\, \in\, \max_j \left(\frac{|\nabla f(x^*)_j|}{|w_j|} \right)\cdot [-1,1] &\text{ if } x_i^* = 0
\end{cases}
\]
exactly characterizes the optimality conditions for \psimple. 
Later, we will generalize this reweighting technique for general atomic sets $\mP_0$, to construct the convex majorant of the general nonconvex problem \eqref{eq:main0}.

\paragraph{Latent group norm.} For the task of selecting a sparse collection of overlapping subvectors, such as in gene identification, the latent group norm was proposed in \cite{obozinski2011group}. For $x\in \R^d$, given a collection of overlapping groups $\mG = \{\mG_1,...,\mG_K\}$ where $\mG_k \subset \{1,...,d\}$, this norm a gauge function, 
\begin{equation}
\kappa_\mP(x) = \|x\|_\mG := \min_{s_k\in \R^d} \;\left\{ \sum_{k=1}^K \|s_k\|_2 : x = \sum_{k=1}^K s_k, \; (s_k)_i = 0\; \forall i\not\in \mG_k\right\}.
\label{eq:latentgroupnorm}
\end{equation}
In particular, \eqref{eq:latentgroupnorm} is the  solution to \eqref{eq:gauge} when 
\[
\mP_0 = \left\{\frac{1}{\sqrt{|\mG_k|}}e_{\mG_k}, \; k = 1,...,K\right\}, \qquad (e_{\mG_k})_i = \begin{cases}
1, & i\in \mG_k\\
0, & \text{else.}
\end{cases}
\]
Then $\sigma_\mP(z) = \max_{k=1,...,K} \; \|z_{\mG_k}\|_2$.
Now consider \eqref{eq:main-convex} for some smooth $\phi$. Then at optimality, decomposing $x^* = \sum_k s_k^*$, for each group $k$,
\[
\begin{cases}
\|z^*_{\mG_k}\|_2 = \phi'(\kappa_\mP(x^*)), & \text{ if } \|s^*_k\|_2 > 0,\\
\|z^*_{\mG_k}\|_2 \leq \phi'(\kappa_\mP(x^*)),  & \text{ if } \|s^*_k\|_2=0.
\end{cases}
\]
Screening in this case refers to identifying the subvectors where, at optimality, $\|s^*_k\|_2$ might be nonzero; however, just as support identification in the 1-norm case does not imply that the values of $x_i^*$ are known, in a similar vein here it does not imply that the values of $s_k^*$ are known.

\paragraph{Gauges over infinite atomic sets.} A popular application of CGM in low-rank matrix completion leverages the gauge properties of the nuclear norm. 
Here, $\supp_\mP(X)$ may be infinite and are almost always nonunique; for example, if $X = aa^T+bb^T$, then any $uu^T$  where $u$ is in the range of $[a,b]$ is in the support of $X$. 
However, guaranteed finite-time safe screening is not possible for the nuclear norm in our CGM variants, as each left and right singular vector is arbitrarily close to another  candidate pair--unlike in the case of finite atoms $\mP_0$, there is nothing to ``snap to". 


\paragraph{Total variation (TV) norm.} We now investigate a case where $\mP_0$ is not finite, but is a union of a finite set and a recession cone, This occurs in  image processing, where a signal is treated as sparse with respect to edges, via the gauge penalty
\[
\kappa_\mP(x) = \|Dx\|_1, \qquad D = \bmat I & 0 \emat - \bmat 0 & I \emat 
\in \R^{d-1,d}
\]
and is the solution to \eqref{eq:gauge} when $\mP_0 = \{b_1,...,b_{d-1}\} \cup \{c \mb 1: c \in \R\}$
where  for $\mb 1$ the all-ones vector,
\[
b_k = \beta_k - \frac{1}{n}\beta_k^T \mb 1, \qquad \beta_k = (\underbrace{1,1,...1}_{k},\underbrace{0,0,..0}_{n-k})\in \R^n.
\]
The difficulty for CGM is that the support function now has a limited domain; specifically, 
\[
\sigma_\mP(z) = 
\begin{cases}
\|u\|_\infty & \text{ if }  z = D^Tu\\
+\infty & \text{ else}  
\end{cases}
\]
and thus the LMO is not always defined.
Note here that if $z\in \range(D^T)$, then $u = D(D^TD)^{-1} z$ is uniquely determined.
Now suppose at optimality, $-\nabla f(x^*) = D^Tu^*$ for some $u^*$. Then  \eqref{eq:alignment} is satisfied when
\[
-\nabla f(x^*)^Tx^* = (u^*)^TDx^* = \underbrace{\|u^*\|_\infty}_{\sigma_\mP(-\nabla f(x^*))} \underbrace{\|Dx^*\|_1}_{\kappa_\mP(x^*)}
\]
and the conditions are very similar to that in the $\ell_1$ case:
\[
|(D x^*)_i| \neq 0 \Rightarrow |u_i^*| =\|u\|_\infty.
\]
Suppose, however, that $-\nabla f(x^*) \not\in\range(D^T)$. Then $x^*$ cannot be optimal, as there exists a direction that decreases $f(x)$ but has no impact on $\kappa_\mP(x)$. Therefore, while this case is problematic for a generalized analysis (and for the intermediary steps of CGM and its variants), it never occurs at optimality.

\paragraph{Gauges with directions of recessions.}
The \emph{recession cone of $\mP$} \citep{rockafellar1970convex,borweinlewis} is defined as
\[
\rec(\mP) = \{r : cr \in \mP \; \forall c \geq 0\}.
\]
Whenever $\mP$ has a direction of recession, CGM struggles as the LMO can return an infinite atom. 
We offer to isolate optimization over this set separately. In particular, suppose 
\[
\mP_0 = \mP_0' \cup \mK, \qquad \mP_0' \cap \mK = \emptyset,
\]
where $\mP_0'$ is a finite set, and thus defining $\mP$ as the convex hull of $\mP_0'$ ensures that $\mP$ is compact. Then we rewrite \eqref{eq:main-convex} as 
\begin{equation}
\minimize{x\in \cone(\mP), y\in \mK}\; f(x+y) + \phi(\kappa_\mP(x)).
\label{eq:main-convex-recession}
\end{equation}
At each iteration, $x$ takes a conditional gradient step, and $y$ is updated through a full minimization. In the case of the TV norm, this is a  small addition, as $y\in \mK$ is implicitly expressed as $y = c\mb 1$, and optimizing over $c$ is just a one-dimensional convex minimization problem.
Since the portion of the solution in $\mK$ is minimized exactly at each step, from this point on we only consider the support recovery properties for recovering the atoms in $\mP_0'$.

\begin{assumption}[Atomic set conditions]
$\mP_0 =\mP'_0 \cup \mK$ where $\mP'_0$ is a finite set of atoms and $\mK$ is the recession cone; moreover, $\mP'_0\cap \mK = \emptyset$. We denote $\mP = \conv(\mP_0')$.
\label{asspt:atoms}
\end{assumption}
Table \ref{tab:gauges} summarizes these examples and key properties. 
Gauges and support functions for convex sets are fundamental objects in convex analysis, and are discussed more by \citet{rockafellar1970convex,borweinlewis,freund1987dual,friedlander2014gauge}.

\begin{table}
\begin{center}
{\small
\begin{tabular}{|cccc|}
\hline
Gauge $\kappa_\mP(x)$  &  Atoms $\mP_0$  & Support fn $\sigma_\mP(z)$ &  LMO$(z)$ \\\hline&&&\\
1-norm & $\left\{\pm e_1,...,\pm e_d\right\}$&  $\|z\|_\infty$ & $\sign(z_k)e_k$,  \\
$\|x\|_1$  &&& $k = \argmax{k}\;|z_k|$ \\
 &&&\\
Mapped 1-norm & $\left\{\pm p_1,...,\pm p_d\right\}$&  $\|Pz\|_\infty$ &  $\sign(p_k^Tz)p_k$,\\
$\|P^{-1} x\|_1$    &  &   &   $k = \argmax{i}\;|p_i^Tz|$\\
 &&&\\
Group norm & $\bigg\{\frac{1}{\sqrt{|\mG_1|}}e_{\mG_1},...,$
&$\max_k \|z_{\mG_k}\|_2$
& $\frac{1}{\sqrt{|\mG_k|}}e_{\mG_k}$,  \\
$\displaystyle\sum_{i=1}^K\|x_{\mG_i}\|_2$,    &   $\qquad\qquad\frac{1}{\sqrt{|\mG_K|}}e_{\mG_K}\bigg\}$ & &  $k = \argmax{k}\;\|z_{\mG_k}\|_2$ \\
 &&&\\
TV norm & $ \{b_k\}_{k=1}^d\cup \;\{c \mb 1: c \geq 0\}$ & $\|D^\dagger z\|_\infty $ if $z\in \range(D^T)$,  & $b_k$,  \\
$\|Dx\|_1$
 & $\beta_k = (\mb 1_k, \mb 0_{n-k})$    & $+\infty$ else.  & $k = \argmax{i} \; |(D^\dagger z)_i|$\\
 &$b_k = \beta_k - \frac{1}{n}\beta_k^T\mb 1$ &&\\
 &&&\\
\hline
\end{tabular}
}
\end{center}
\caption{Common norms, their atoms, support functions, and their LMOs. In particular, computing each LMO is computationally cheap, especially compared to computing the proximal operator of the gauge,or even the gauge itself.}
\label{tab:gauges}
\end{table}

%

\subsection{Generalized smoothness}



To ensure the uniqueness of $\dsupp_\mP(-\nabla f(x^*))$ and to give a useful gap bound, we again need a notion of smoothness on $f$. 
\begin{definition}
A function $f:\R^d\to\R$ is  $L$-smooth 
 with respect to $\mP$ if for all $x,y \in \R^d$:
\begin{equation}
f(x) - f(y) \leq \nabla f(y)^T(x-y) + \frac{L}{2}\kappa_{\mP}(x-y)^2.
\label{eq:ass:smoothness}
\end{equation} 
\end{definition}
The purpose of this generalized notion is that sometimes, given the data, tighter bounds can be computed~\citep[see, e.g.,][]{nutini2015coordinate}.
It is similar in spirit to the notion of \emph{relative smoothness} \citep{bauschke2017descent,lu2018relatively} which facilitate the analysis of generalized proximal gradient methods, where the 2-norm squared proximity measure is replaced by a Bregman divergence. For CGM, it is  more computationally efficient to consider generalized gauges as the penalty generalization, which we incorporate to the generalized smoothness definition. Additionally, the subadditivity property of gauges assists with bounding the iterates, a crucial step in the convergence proof.

\begin{assumption}[Generalized smoothness]
The convex function $f$ is $L$-smooth w.r.t. $\widetilde\mP := \mP \cup (-\mP)$.
\label{asspt:gensmooth}
\end{assumption}

\paragraph{Example: Quadratic function.} Suppose that 
$f(x) = \frac{1}{2} \|Ax\|_2^2 + b^Tx$.
Then 
\[
L = 
\begin{cases}
L_1 := (\max_i \|A_{:,i}\|_2)^2, & \kappa_\mP = \|\cdot\|_1\\
L_2 := \|A\|_2^2, & \kappa_\mP = \|\cdot\|_2\\
L_\infty := (\sum_i \|A_{:,i}\|_2)^2, & \kappa_\mP = \|\cdot\|_\infty.
\end{cases}
\]
While norm bounds would give  ${d^2L_{1} \geq  dL_{2} \geq  L_{\infty}}$, the actual values in $A$ might lead to tighter inequalities.

\paragraph{Relationship to usual smoothness.} Suppose that $f$ is $L_2$-smooth in the usual sense (with respect to $\|\cdot\|_2$).  Then since $\diam(\mP)\kappa_\mP(x) \geq \|x\|_2$, it follows that $L \leq \diam(\mP) L_2$. In this way, we refine the analysis of CGM by absorbing the usual ``set size'' term into $L$, which in certain cases may be  smaller than $\diam(\mP) L_2$. 

\begin{property}[Uniqueness of gradient] If \eqref{eq:ass:smoothness} holds and $0\in\inte~\mP$, then $\nabla f(x^*)$ is unique at the optimum.
\label{cor:uniquegrad-convex}
\end{property}
See appendix \ref{app:smoothness} for proof.

\subsection{Duality and gap}
For $\phi$ is monotonically nondecreasing, the convex function $h(x) = \phi(\kappa_\mP(x))$
has conjugate $h^*(z) = \phi^*(\sigma_{\mP}(z))$.
Then, rewriting  \eqref{eq:main-convex-recession} gives the primal-dual pair
\[
\pconvex\quad 
\begin{array}[t]{ll}
\mini{x,y,w}& f(w) + \phi^*(\kappa_{\mP}(x)) \\
\st & w = x+y,\; y\in \mK
\end{array}
\qquad\dconvex\quad
\begin{array}[t]{ll}
    \maxi{z} & -f^*(-z) - \phi^*(\sigma_{\mP}(z))\\
\st & z\in \mK^\circ
\end{array}
\]
where $\mK^\circ$ is the polar cone of $\mK$.
Then the duality gap between \pconvex~and \dconvex~can be written as 
\begin{eqnarray*}
\gap(x,y,z) = f(x+y) + h(x) + f^*(-z)-h^*(z) +  \iota_{\kappa^\circ}(z)
\end{eqnarray*}
where $\iota_{\mK^\circ}(z) = +\infty$ if $z$ is not dual-feasible, and 0 otherwise.

\begin{property}[Feasible gradient]
\label{prop:feasgrad}
Take $z := -\nabla_x f(x+y)$. Then $z  = -\nabla_y f(x+y)$. Additionally, if $y = \argmin{y'\in \mK}\; f(x+y')$ then $z\in \mK^\circ$.
\end{property}
\begin{proof}
The first part is true from  chain rule. Then, since
\[
(\mN_\mK(y))^\circ = \underbrace{\mT_\mK(y)}_{\text{tangent cone}} \supseteq \mK = (\mK^\circ)^\circ,
\] 
then from optimality conditions, $z\in \mN_{\mK}(y)\subseteq \mK^\circ$.
\qed
\end{proof}
From Property \ref{prop:feasgrad},  the LMO step acquires $s$ where for $z := -\nabla_x f(x+y)$,
\[
-z^Ts + h(s) = \min_{s'} \; -z^Ts' + h(s) = -h^*(z).
\]
Additionally, by Fenchel-Young's inequality, we know that $f(x) + f^*(\nabla f(x)) = \nabla f(x)^Tx$, and thus
we can  simplify the gap to an online-computable quantity
\[
\gap(x,y,\nabla_x f(x+y)) = -\nabla_x f(x+y)^T(s-x) + h(x) - h(s).
\]

\begin{property}[Gap bounds gradient error]
Given a primal feasible $x$ and  denote the  optimum variable as
\[
 x^* = \argmin{x'}\; \min_{y\in \mK} f(x+y) + h(x).
\]
Furthermore, denote $y = \argmin{y'\in \mK}\; f(x+y')$ and $y^* = \argmin{y'\in \mK}\; f( x^*+y')$.
 Then the duality gap
 bounds the gradient error
\begin{equation}
\gap_\mP(x,y,-\nabla f(x+y)) \geq \frac{1}{2L}\sigma_{\widetilde\mP}(\nabla f(x+y) - \nabla f(\bar x^*+\bar y^*))^2.
\label{eq:gradsuboptbound}
\end{equation}
\label{prop:suboptbound}
\end{property}

\begin{proof}  
Since the  conjugate of $h(x) = \phi(\kappa_\mP(x))$ is $h^*(z) = \phi^*(\sigma_\mP(z))$, then 
\begin{equation}
\phi^*(\sigma_\mP(z)) = \sup_{x}x^Tz -\phi(\kappa_\mP(x)) \geq (x^*)^Tz-\phi(\kappa_\mP(x^*)). 
\label{eq:suboptbound-helper1}
\end{equation}
Then denoting $z =-\nabla f(x+y)$,
\begin{eqnarray*}
\gap_\mP(x,y, z)&=& \underbrace{f(x+y)+ f^*(-z)}_{\text{Fenchel Young}} + \phi(\kappa_{\mP}(x))  + \underbrace{\phi^*(\sigma_{\mP}(-\nabla f(x)))}_{\eqref{eq:suboptbound-helper1}}, \\
&\geq  & (x^*-x)^Tz + \underbrace{y^T\nabla f(x)}_{y\in \mK, \nabla f(x)\in \mK^\circ}  +  \underbrace{\phi(\kappa_{\mP}(x))  - \phi(\kappa_{\mP}(x^*))}_{\text{convexity of $h$}}, \\
&\overset{(\star)}{\geq} & (x-x^*)^T(z^*-z) \geq \frac{1}{2L}\sigma_{\widetilde\mP}(z^*-z)^2
\end{eqnarray*}
where $(\star)$ is from picking $-\nabla f(x^*+y^*)\in \partial h(x^*)$.
\end{proof}

\subsection{Invariance}
One appealing feature of the CGM is that the iteration scheme and analysis can be done in a way that is 
invariant to both linear scaling and translation. However when the gauge function is not used as an indicator, this translation invariance vanishes; in general, $\kappa_\mP(x) \neq \kappa_{\mP+\{b\}}(x+b)$. Therefore the generalized problem formulation \eqref{eq:main-nonconvex} is only linear (not translation) invariant.

\paragraph{Example.} Consider $\kappa_\mP(x) = \|x\|_1$ for $x\in \R^2$. take specifically $x = (-1,-1)$ and $b = (1,1)$. Then $\kappa_\mP(x) = 2$, but $\kappa_{\mP+\{b\}}(x+b) = \kappa_{\mP+\{b\}}(0) = 0\neq 2$.

\paragraph{Invariance.} Define $\mQ = A\mP$, and $f(x) = g(Ax)$. Define $w = Ax$ where $A$ has full column rank. Then, using \eqref{eq:gauge-classical} and chain rule, the following hold
\begin{itemize}[label={$\boldsymbol{\cdot}$}]
\item $f(x) = g(w)$ and $\nabla f(x) = A^T\nabla g(w)$,
    \item 
$\kappa_\mP(x) = \kappa_\mQ(w)$ and $\sigma_{\mP}(-\nabla f(x)) 
=\sigma_{\mQ}(-\nabla g(w))$,



\item $\LMO_{\mQ}(-\nabla g(w)) = A\;\LMO_{\mP}(-\nabla f(x)) $,

\item 
 $f(x) = g(Ax+b)$ is $L$-smooth w.r.t. $\mP$ iff $g$ is $L$-smooth w.r.t. $ \mQ$.
\end{itemize}

\section{RP-CGM for general nonconvex sparse optimization}
\label{sec:nonconvex}

Finally, we consider RP-CGM on the general optimization problem
\begin{equation}
\mini{x}  \;  f(x) + \underbrace{\phi(r_\mP(x))}_{h(x)},
\qquad
r_\mP(x) = \left\{\min_{c_p\geq 0}\; \sum_{p\in \mP_0} \gamma(c_p) : \sum_{p\in \mP_0} c_p p = x\right\}.
\label{eq:main-nonconvex}
\end{equation}
By imposing the concave transformation on $c_p$, we effectively gain the same effect as the nonconvex regularizer on the $\ell_1$ norm in section \ref{sec:simple}. For the most part, much of the analysis will seem very similar to that in section \ref{sec:simple}, especially in the proofs of key concepts, which we therefore put in the appendix to avoid repetitiveness. We also use much of the same assumptions (\ref{asspt:phigamma1}, \ref{asspt:Lsmooth-l1}, \ref{asspt:phi2}) and analyses for the scalar functions $\gamma$ and $\phi$.

\subsection{Support recovery}
As it was for $\kappa_\mP$, the domain of $r_\mP$ is $\cone(\mP)$.
However, the support of $\kappa_\mP(x)$ and $r_\mP(x)$ are often not equivalent.

\paragraph{Example: Different optimal support.}
Consider $\kappa_\mP(x) = \|x\|_1$ and $r_\mP(x) = \frac{1}{\sqrt{2}}\sum_i \sqrt{|x_i|}$.
The constrained optimization problem 
\[
\minimize{x} \quad f(x) := -4x_1-3x_2-4x_3 \qquad
\subjectto \quad \kappa_\mP(x) \leq 1
\]
has optimal solution $x^* = (1/2,0,1/2)$. We verify this from the normal cone condition, where
\[
\nabla f(x^*)^T(x-x^*) \geq  -\underbrace{\|\nabla f(x^*)\|_\infty\|x\|_1}_{\leq 4} + 4 \geq 0.
\]
Note that $r_\mP(x^*) = 1$ as well. However, taking $\bar x = (0, \sqrt{2}, 0)$ also yields $r_\mP(\bar x) = 1$, and has a lower objective value 
\[
f(\bar x) = -3\sqrt{2} \approx -4.24 < -4 = f(x^*).
\]

\paragraph{Example: Different gauge support.} The problem can be made even worse, in that the support of $x$ w.r.t. $r_\mP$ may not even intersect with that w.r.t.  $\kappa_\mP$. Suppose that 
\[
\mP_0 = \left\{\bmat 1\\1\emat, \bmat 0\\3 \emat, \bmat 3\\0 \emat\right\}
\]
and consider $x = (6,6)$. Then, taking  $\gamma(c) = \sqrt{c}$,  we have two options
\[
x = (0,3) + (3,0), \qquad \kappa_\mP(x) = 2, \quad r_\mP(x) = 2\sqrt{2} \approx 2.8,
\]
\[
x =6 \cdot (1,1), \qquad  \kappa_\mP(x) = 6,\quad r_\mP(x) = \sqrt{6} \approx 2.4.
\]
In other words, the support $\supp_\mP(x)$ as defined in \eqref{eq:convex-gauge-def} may not be the support created by the nonconvex gauge $r_\mP(x)$, which is often sparser. More generally, $r_\mP(x)$ does not act merely as a concave transformation on the weights $c_p$ in $\kappa_\mP$, as even the atoms themselves may be selected differently. However, it is worth noting that this scenario does not happen for the $\ell_1$ norm or the TV norm, which have unique and consistent supports across choices of monotonically increasing $\gamma$.


\subsection{Stationary points}
We can rewrite \eqref{eq:main-nonconvex}, as the combined optimization problem over $c_p$, $p\in \mP_0$:
\begin{equation}
\minimize{c_p\geq 0} \quad f\bigg( \sum_{p\in \mP_0} c_p p\bigg) + \phi\bigg(\underbrace{\sum_{p\in \mP_0}\gamma(c_p)}_{=:\xi}\bigg).
\label{eq:main-nonconvex-reform}
\end{equation}
The stationary points of \eqref{eq:main-nonconvex-reform} are $x$ satisfying
\begin{equation}
\forall p\in \mP_0: \;0 \in \nabla f(x)^Tp +  \phi'(\xi)  \, \partial \gamma(c_p), \quad   \text{ at } x= \sum_{p\in \mP} c_p p.
\label{eq:nonconvex-optsupport}
\end{equation}
Our goal is to find a  support of such a stationary point $x^*$.
Given $\gamma$ smooth everywhere except at 0, note the close similarity between this and the support optimality conditions for convex gauges:
\begin{eqnarray*}
c_p > 0 \; &\Rightarrow& \; -p^T\nabla f(x^*) =   \alpha  \;  \gamma'(c_p) \quad \text{(no wiggle room)},\\
c_p = 0 \; &\Rightarrow& \; -p^T\nabla f(x^*) \in  \alpha\;  [-\infty, \gamma_{\max}] \quad \text{(wiggle room exists)}.
\end{eqnarray*}
Here, the wiggle room condition looks asymmetric, but note that if $p$ and $-p$ is in $\mP_0$, then $c_p = c_{-p} = 0$ implies $-p \nabla f(x^*) \in \alpha [-\gamma_{\max},\gamma_{\max}]$, recovering the symmetric condition from secion \ref{sec:simple}.
As before, since $\gamma'$ is a decreasing function, a nonzero coefficient for $x^*$ does not mean a maximal gradient inner product. 


\subsection{RP-CGM}
In the case that $\mP_0$ includes directions of recession, we again treat them separately; given assumption \ref{asspt:atoms}, $\mP_0 = \mP_0' \cup \mK$ and $\mP_0'$ is finite. 
We define the \emph{reweighted atomic set} for a given reference point $x$ as 
\[
\mP_0(x) = \left\{\frac{1}{\gamma'(\coeff_\mP(x,p))}p : p\in \mP'_0\right\},\quad \mP(x) = \conv(\mP_0(x)).
\]
Then
$r_\mP(s;x) = \kappa_{\mP(x)}(s)$, 
with corresponding reweighted support function
\begin{equation}
\sigma_{\mP(x)}(z) = \max_{p\in \mP_0}\; \frac{p^Tz}{\gamma'(\coeff_\mP(x,p))}.
\label{eq:reweighted-support}
\end{equation}
At each iteration, we take a penalized conditional gradient step toward solving the reweighted gauge optimization problem with dual
\begin{eqnarray*}
    \pgeneral 
    &&
    \minimize{x,y\in \mK} \quad f(x+y) + \phi(r_0 + \kappa_{\mP(\bar x)}(x))\\
    \dgeneral
    &&
    \maximize{z\in \mK^\circ} \quad -f^*(-z) - \phi^*(\sigma_{\mP(\bar x)}(z))+r_0\cdot \sigma_{\mP(\bar x)}(z)
\end{eqnarray*}
A description of the most generalized version of the reweighted method is given in Algorithm \ref{alg:general-nonconvex}.

\begin{algorithm}
\caption{RP-CGM on general nonconvex sparse optimization}
\begin{algorithmic}[1]

\Procedure{RP-CGM}{$f$, $\phi$, $\gamma$, $\mP_0 = \mP_0' \cup \mK$,  max iter $T$}       
    \State Initialize with any $x^{(0)}\in \cone(\mP)$ where $\mP$ is the convex hull of $\mP_0'$, $y^{(0)} \in \mK$
    \State 
    \For{$t = 1,...T$}
        \State Compute the projected negative gradient $z = -\nabla f(x^{(t)}+y^{(t)})$
        \State Compute the reweighted atomic set $\mP(x)$
        
        \State Compute next atom $s =\xi p$ where \Comment{Pick next atom}
        \[
        p = \LMO_{\mP(x)}(z), \quad \xi = (\phi^*)'(\sigma_\mP(z))
        \]
        
        \State Update $x^{(t+1)}$ \Comment{Merge}
        \[
        x^{(t+1)} = (1-\theta^{(t)})x^{(t)} + \theta^{(t)} s, \quad \theta^{(t)} = 2/(1+t)
        \]
        
        \State Update $y^{(t+1)}$ using exact minimization \Comment{Recession component}
        \[
        y^{(t+1)} = \argmin{y\in \mK}\; f(x^{(t+1)} + y)
        \]
    \EndFor
\Return  $x^{(T)} + y^{(T)}$
\EndProcedure

\end{algorithmic}
\label{alg:general-nonconvex}
\end{algorithm}

\subsection{Convergence}


\begin{property}[Residual]
Denoting  $\res_\mP(x) = \gap_\mP(x,-\nabla f(x);x)$ the gap at $x$ with reference $x$, then
 \[
 \res_\mP(x)  \geq 0 \; \forall x, \qquad \res_\mP(x) =0 \iff \text{ $x$ is a stationary point of \ref{eq:main-1norm}.}
 \]
 \label{prop:gapres-general}
\end{property}
The proof follows closely that of Property \ref{prop:gapres-l1}; see appendix \ref{app:convergence} for full details.

\begin{lemma}[One step descent]
 Suppose $f$ is $L$-smooth w.r.t. $\mP$ (unweighted). Take 
\[
x^+ = (1-\theta) x + \theta s, \qquad s = \argmin{\tilde s} \; \nabla f(x+y)^T\tilde s +  \bar h(s;x)
\]
for some $\theta \in (0,1)$. Define $y = \argmin{y\in \mK} \; f(x+y)$,\;  $y^+ = \argmin{y\in \mK}\; f(x+y)$.
Then  
\[
f(x^++y^+) + h(x^+) - f(x+y)-h(x) \leq -\theta \res(x) + \frac{L\theta^2}{2} \kappa_\mP(s-x)^2.
\]
\label{lem:onestepdescend1}

\end{lemma}
\begin{proof} 
From $L$-smoothness we have 
\begin{eqnarray}
f(x^++y^+)-f(x+y) &\leq& f(x^++y)-f(x+y)\nonumber\\
&\leq&\nabla f(x+y)^T(x^+-x) + \frac{L}{2} \kappa_\mP(x^+-x) \nonumber\\
&=& \theta \nabla f(x+y)^T(s-x) + \frac{L\theta^2}{2}\kappa_\mP(s-x)^2
\label{eq:stephelper1}
\end{eqnarray}
Denote $ \nu =  \sigma_{\widetilde\mP(x)} (-\nabla f(x+y))$.
Since $s = \xi \phi'(\nu)$, then 
\begin{eqnarray}
\nabla f(x+y)^Ts + \phi(r_0 + \bar r_\mP(s;x)) &=& \min_{\tilde s} \underbrace{\nabla f(x+y)^T \tilde s}_{=-\xi\cdot\nu} +  \phi\bigg(r_0 +  \underbrace{\bar r_\mP(s;x)}_{=\xi}\bigg) \nonumber\\
&=& \nu r_0 - \phi^* (\nu).
\label{eq:stephelper2}
\end{eqnarray}
Also,  by definition of residual,
\begin{eqnarray}
\res_\mP(x) &=& f(x+y) + f^*(\nabla f(x+y)) +\phi(r_\mP(x+y)) + \phi^*(\nu) - r_0\cdot \nu\nonumber\\
&= & \underbrace{\nabla f(x+y)^T(x+y)}_{\nabla f(x+y)^Ty\geq 0} +\phi(r(x)) + \phi^*(\nu) - r_0\cdot \nu\nonumber\\
&\geq  & \nabla f(x+y)^Tx +\phi(r(x)) + \phi^*(\nu) - r_0\cdot \nu.
\label{eq:stephelper3}
\end{eqnarray}

Therefore taking $F(x+y) = f(x+y) + \phi(r(x))$ and combining \eqref{eq:stephelper1}, \eqref{eq:stephelper2}, and \eqref{eq:stephelper3},
\begin{eqnarray*}
 F(x^++y^+)- F(x+y) 
&=&-\theta   \res(x) + \theta \left(\phi(r_\mP(x)) - \phi(r_0 + \bar r_\mP(s;x))  \right) \\
&& \qquad + \frac{L\theta^2}{2}\kappa_\mP(s-x)^2 + \phi( r_\mP(x^+)) - \phi( r_\mP(x)) \\
\end{eqnarray*}
Next, by convexity of $\phi$,
\begin{eqnarray*}
 (1-\theta)\phi( r_\mP(x))  +  \theta \phi(r_0+ \bar r_\mP(s;x)) &\geq& 
 \phi(r_\mP(x) +\bar r_\mP(x^+;x)-\bar r_\mP(x;x) )\\
 &\overset{\text{majorant}}{\geq}& \phi(r_\mP(x^+;x))
\end{eqnarray*}
which leaves the desired result.
\end{proof}

\begin{lemma}[Iterate gauge control.]
Suppose additionally $\theta^{(t)} = 2/(t+1)$. Then
\begin{eqnarray*}
 \kappa_\mP(s^{(t)}-x^{(t)}) 
  &\leq& \frac{\gamma_{\max}}{\gamma_{\min}\mu}\bigg(2\sigma_{\widetilde\mP}(\nabla f(x^*+y^*)) +
 \sqrt{2L\Delta^{(t)}} +
  \frac{2}{t(t-1)}\sum_{u=1}^{t-1} \sqrt{2L\Delta^{(u)}}\bigg)  \\
  &&\qquad + 2\nu_0\gamma_{\max}+\kappa_\mP(x^{(0)}) .
\end{eqnarray*}

\label{lem:onestepdescend2}
\end{lemma}

\begin{proof}
\begin{eqnarray*}
 \kappa_\mP(s^{(t)}-x^{(t)}) &\overset{\substack{\text{subadditive}\\\text{gauge}}}{\leq} &  \kappa_\mP(s^{(t)})+\kappa_\mP(x^{(t)})\\
&\overset{\text{convexity}}{\leq}& \kappa_\mP(s^{(t)}) +  \theta^{(t-1)} \kappa_\mP(s^{(t-1)} )  +(1-\theta^{(t-1)})\kappa_\mP(x^{(t-1)}) \\
&\overset{\text{recursion}}{\leq}& \kappa_\mP(s^{(t)}) + \kappa_\mP(x^{(0)}) + \sum_{u=1}^{t-1} \theta^{(u)} \underbrace{\prod_{u'=u+1}^{t-1}(1-\theta^{(u')})}_{=\frac{(u+1)u}{t(t-1)}}\kappa_\mP(s^{(u)} )  \\
 &\leq& \kappa_\mP(s^{(t)}) +\kappa_\mP(x^{(0)}) +  \frac{2}{t(t-1)}\sum_{u=1}^{t-1} u \kappa_\mP(s^{(u)}).
 \end{eqnarray*}
In general, for any $x$, $z$,  $\bar x$, 
\[
\kappa_{\mP}(x) \leq \gamma_{\max}\kappa_{\mP(\bar x)}(x), \qquad 
\sigma_{\mP}(z) \geq \frac{1}{\gamma_{\min}}\sigma_{\mP(\bar x)}(z).
\]
Taking $y^{(u)} = \argmin{y\in \mK}\;f(x^{(u)} + y)$, $z^{(u)} = -\nabla f(x^{(u)}+y^{(u)})$,  ${z^* = -\nabla f(x^*+y^*)}$:
\begin{eqnarray*}
\kappa_{\mP(\bar x)}(s^{(u)}) 
=
(\phi^*)'\left( \sigma_{\mP(\bar x)}(z^{(u)}) \right)
&\overset{\text{Asspt \ref{asspt:phi2}}}{\leq} & \mu^{-1} \cdot  \sigma_{\mP(\bar x)}(z^{(u)}) + \nu_0\\
&\overset{\text{Bound on $\gamma'$}}{\leq} & \frac{1}{\mu r_{\min}} \sigma_{\mP}(z^{(u)}) + \nu_0\\
&\overset{\substack{\text{$\Delta$-ineq +}\\\text{ Prop. \ref{prop:res-bound-graderr}}}}{\leq} & \frac{1}{\mu r_{\min}}\left( \sigma_{\widetilde \mP}(z^*)  + \sqrt{2L\Delta^{(u)}}\right) + \nu_0.
\end{eqnarray*}
Putting it all together  gives the desired result.
\qed
\end{proof}
From Lemmas \ref{lem:onestepdescend1} and \ref{lem:onestepdescend2}, we arrive at
\[
\Delta^{(t+1)}-\Delta^{(t)} \leq -\theta^{(t)}   \res_\mP(x^{(t)}) + (\theta^{(t)})^2\left( B\Delta^{(t)}  + B  \bar \Delta^{(t-1)} 
  + A \right) 
\]
for constants 
\[
A = \left(\frac{6L \gamma_{\max}^2}{\mu^2\gamma_{\min}^2 }\sigma_{\widetilde\mP}(-\nabla f(x^*+y^*)+6\gamma_{\max}\nu_0+3\kappa_\mP(x^{(0)}) \right)^2, \qquad B = \frac{3L^2\gamma_{\max}^2}{\mu^{2}\gamma_{\min}^2}
\]
and where $\bar \Delta^{(t)}$ is defined as an averaging over square roots, e.g. 
\[
\sqrt{\bar\Delta^{(t)}} =\frac{2}{t(t+1)} \sum_{u=1}^{t} u \sqrt{\Delta^{(u)}}. 
\]

\begin{thm}[Convergence]
Consider $G$ large enough such that for all $t < 6B$, $\Delta^{(t)} t \leq G$ and $G > 24A$. 
Given assumptions \ref{asspt:phigamma1}, \ref{asspt:phi2}, \ref{asspt:atoms}, \ref{asspt:gensmooth},  with iterates  $x^{(t)}+y^{(t)}$ from algorithm \ref{alg:general-nonconvex}, using   $\theta^{(t)} = 2/(t+1)$, then
\[
\Delta^{(t)}  \leq \frac{G}{t+1} \qquad \text{ and } \qquad
\min_{i\leq t} \res(x^{(i)}) \leq \frac{3G}{2\log(2) (t+1)}.
\]
\label{th:convergence}
\end{thm}
Given Lemmas \ref{lem:onestepdescend1} and \ref{lem:onestepdescend2}, the details of the proof closely mirror steps in previous works, and thus we give the explicit details in appendix \ref{app:convergence}.

\paragraph{Comparison with CGM} In \cite{jaggi2013revisiting}, the primal convergence rate for vanilla CGM (with noiseless gradients) is given as $\Delta^{(t)} \leq \frac{2 C_f}{t+2}$  where $C_f$ is a curvature constant that depends on the conditioning of $f$ and the size of $P$. These players appear here in the form of the conditioning of $f$ (quadratic in $L/\mu$), and implicitly $\sigma_{\widetilde\mP}$ (which grows proportionally with $\diam~\mP$). The new players $\nu_0$, $\gamma_{\min}$, and $\gamma_{\max}$ account for the penalty and nonconvex generalizations.


\subsection{Invariance}
Consider $\mQ = A\mP$, $f(x) = g(Ax)$, $w = Ax$, $\bar w = A \bar x$, where $A$ has full column rank. Additionally, assume $x, \bar x \in \cone(\mP)$. Then the following hold.
\begin{itemize}
\item \textbf{Penalty.} $r_\mP(x) = r_\mQ(w)$.
This follows from noting that 
\[
x =  \sum_{p\in \mP} c_p p \iff w =  \sum_{p\in \mP} c_p (Ap) =  \sum_{q\in \mQ} c'_q q
\]
and in fact noting that the coefficients are equal ($c'_q = c_{Ap}$).


\item \textbf{Stationarity.} 
We construct $P$ with columns containing the atoms in $\mP_0'$, and $c$ such that
$x = Pc$, $w = Ax = APc$. 
\[
 P^T\partial r_\mP(x)=\partial r_\mP(c) \overset{r_\mP(c) = r_\mQ(c)}{=} \partial r_\mQ(c) =  P^TA^T\partial r_\mQ(w).
\]
Additionally, for any stationary point $x^*$, if $\nabla f(x^*)\not\in \cone(\mP)$ then there exists a descent direction that is uneffected by the penalty $r_\mP(x)$, and thus it must be that $\nabla f(x^*)\in \cone(\mP)$. By the same token, $A^T\nabla g(w^*)\in \cone(\mP)$. Therefore, the stationary conditions are equivalent: for $x^* = Aw^*$,
\[
0\in \nabla f(x^*) + P^T\partial r_\mP(x^*)  \iff 0\in A^T\nabla g(w^*) + A^T\partial r_{\mQ}(w^*). 
\]


\end{itemize}
Additionally, it can be shown through the chain rule that $A\mP(x) = \mQ(w)$ and $\res_\mP(x) = \res_\mQ(w)$.
Overall, this shows that the steps and analysis of RP-CGM are all invariant to linear transformations on $x$.

\subsection{Screening}

We now describe the  gradient error measured in terms of this ``dual gauge"
$\sigma_{\widetilde\mP}(\nabla f(x)-\nabla f(x^*))$ where the symmetrization $\widetilde\mP := \mP \cup -\mP$ ensures that ${\sigma_{\widetilde\mP}(z-z^*)=\sigma_{\widetilde\mP}(z^*-z)}$, bounding errors in both directions.

\begin{property}[Residual bound on gradient error.]
Denote $D(x) = r_\mP(x)-r_\mP(x^*) + \bar r_\mP(x;x) - \bar r_\mP(x^*;x)$
the linearization error at $x$. 
Denoting $x^*$ a stationary point of \eqref{eq:main-nonconvex} and $y(x) = \argmin{y'\in \mK}\; f(x+y')$, then
\begin{multline*}
\sigma_{\widetilde\mP}(\nabla f(x+y(x))-\nabla f(x^*+y(x^*))) \leq\\ \frac{LD(x)}{2\gamma_{\min}} + \sqrt{\frac{L^2D(x)^2}{4\gamma^2_{\min}} + L\res(x) + LD(x) \frac{\sigma_{\widetilde\mP}(\nabla f(x+y(x)))}{\gamma_{\min}}}.
\end{multline*}
\label{prop:res-bound-graderr}
\end{property}
The linearization error $D(x) = 0$ when the regularizer is convex.
The proof is similar to that for Prop. \ref{prop:resboundsgrad-l1}, and is detailed in appendix \ref{app:convergence}.

\begin{thm}[Dual screening]
For any $x$ and some choice of $\epsilon > 0$,  define the screened set as 
\begin{equation}
\mI_\epsilon(x) = \{p\in \mP_0 : \sigma_{\widetilde\mP}(\nabla f(x)) + p^T\nabla f(x) > \epsilon + 2\sqrt{L\res(x)+\epsilon}\}.
\label{eq:gaprule}
\end{equation}
Then given assumptions \ref{asspt:phigamma1},  \ref{asspt:phi2}, and \ref{asspt:gensmooth},
if 
\[
\epsilon \geq \frac{L D(x)}{\gamma_{\min}} \; \max\left\{\frac{1}{2},\frac{L D(x)}{4\gamma_{\min}} + \sigma_{\widetilde \mP}(\nabla f(x))\right\}
\]
then $p\not\in \supp_\mP(x^*)$,
where $x^*$ is the optimal variable in \eqref{eq:main-convex}.
\label{th:screening2}
\end{thm}



In the convex case, $D(x) = 0$, and thus we pick $\epsilon = 0$ in our screening rule. In this scenario, not only does this screening rule achieve finite-iteration support identification, but the finite time $\bar t$ depends directly on $\delta_{\min}$.
\begin{thm}[Support identification of screened P-CGM]
\label{th:supportid}
Given assumptions \ref{asspt:phigamma1}, \ref{asspt:phi2}, \ref{asspt:atoms}, \ref{asspt:gensmooth},  then the screening rule for convex penalties 
\[
\mI^{(0)} = \mP_0, \qquad \mI^{(t)} = \mI^{(t-1)}\setminus \{p \in \mP_0 : p\in \mI_0(x) \mathrm{\; for\;  } x=x^{(t)}\},
\]
is safe and convergent:
\[
 \mI^{(t)}\supseteq \supp_\mP(x^*),\; \forall t, \quad \text{ and } \quad  \mI^{(t)} = \supp_{\mP(x^*)}(x^*), \; t \geq t'
\]
where $t'$ is such that
\begin{equation}
 \sqrt{L\min_{i\leq t'}\res(x^{(i)})} <  \delta_{\min}/3
\label{eq:suppID:bound}
\end{equation}
which happens at a rate $t' = O(1/(\delta_{\min}^2))$.
\end{thm}
\begin{proof}
This is a direct consequence to theorems \ref{th:convergence} and \ref{th:screening2}.
\qed
\end{proof}

Note that Theorem \ref{th:supportid} imposes no conditions on the sequence $\theta^{(k)}$, or choice of $\phi$, $f$, etc.,  except $L$-smoothness of $f$. In other words, for any method where the gap is easily computable and its convergence rate known, then a corresponding screening rule and support identification rate automatically follow.
Additionally, computing $L$ may be challenging, depending on $\kappa_\mP$; as shown previously, at the very least it may require a full pass over the data. However, this is a one-time calculation per dataset, and can be estimated if data are assumed to be drawn from specific distributions (as in sensing applications). 


\section{Experiments}
\label{sec:experiments}

We now compare P-CGM and RP-CGM on the sensing problem of recovering an element-wise sparse variable (Figure \ref{fig:sensing_smallm}). More extensive numerical results are given in appendix \ref{app:experiments}.

We solve a least squares problem
\begin{equation}
\minimize{x} \qquad \tfrac{1}{2m}\|Ax-b\|_2^2 + \phi(r_\mP(x)).
\end{equation}
where $A\in \R^{m\times n}$ as $A_{ij}\sim\mN(0,1/n)$ i.i.d. for $i = 1,...,m$, $j = 1,...,n$, and for a given $x_0$ and $\eta$, generate $b_i \sim \mN(\sum_j A_{ij} (x_0)_j, \eta)$ i.i.d.  
It is clear that given the same $\lambda$, RP-CGM is more aggressive, even with mild choices of $\theta$ and $\xi$. However,  better overall sensing (higher F1 score) requires hyperparameter tuning. 

\begin{figure}
    \centering
    \includegraphics[width=2in]{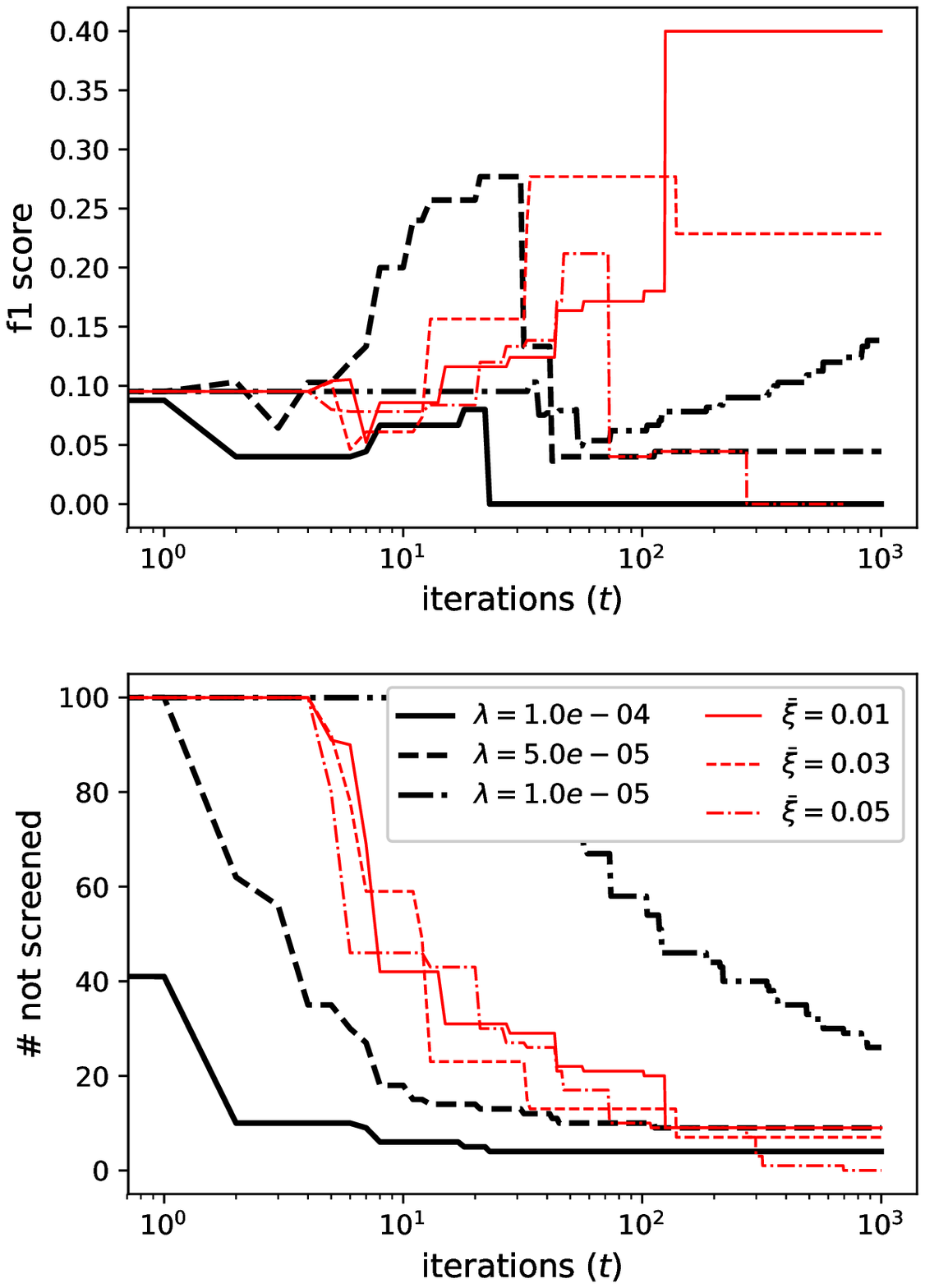}
    \includegraphics[width=2in]{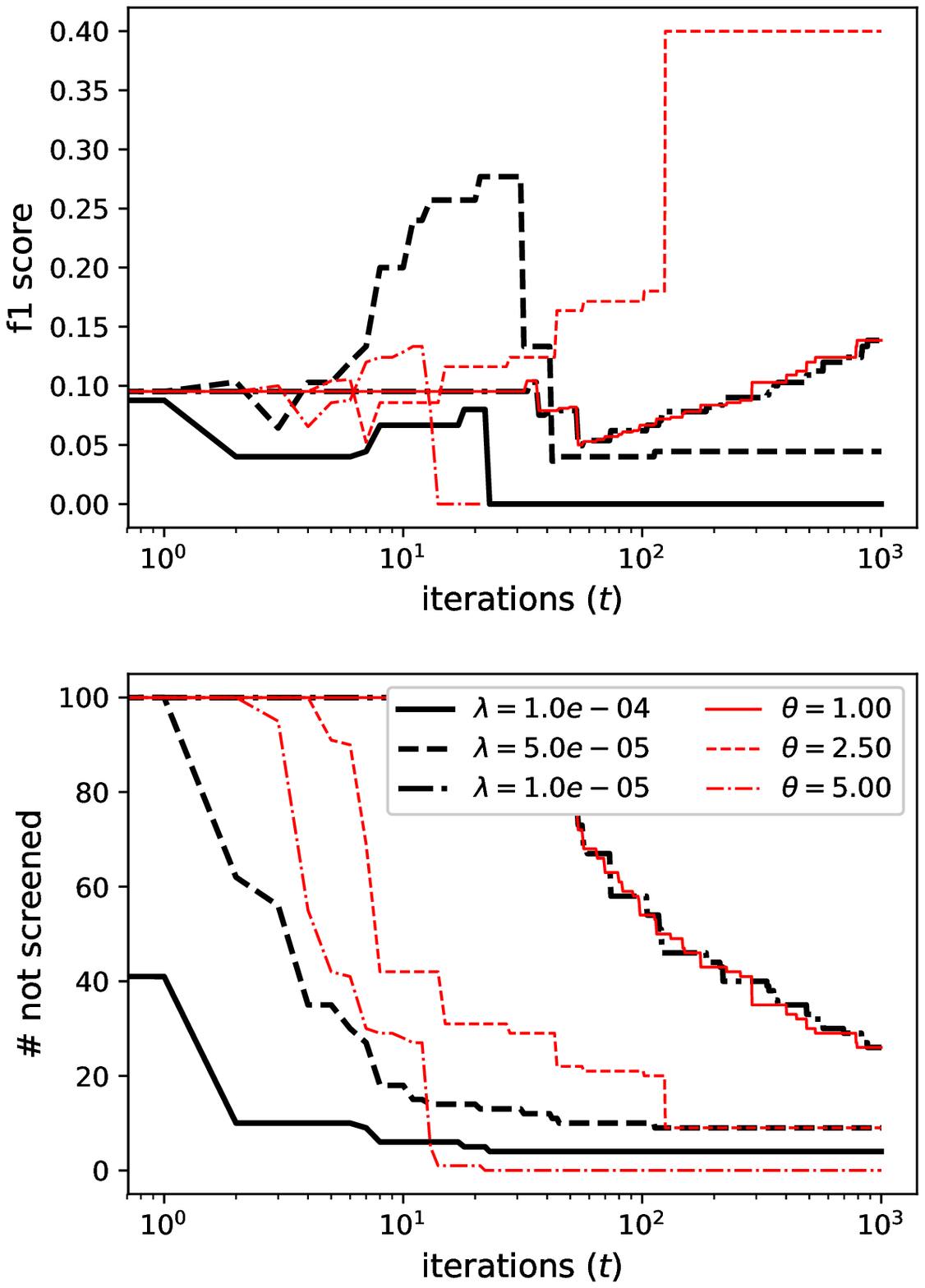}
    \caption{Sensing ($m = 100$, $\eta = 100$). Default values are $\lambda = 0.00001$, $\theta = 2.5$, $\bar \xi = 0.01$. 
    We take  $n = 100$, and $x_0$ has 5 nonzeros  with values i.i.d. $\mN(0,1)$.  Here, $\phi(\xi) = \tfrac{1}{2}\xi^2$, $\gamma$ is as described in \eqref{eq:defgamma:exp}, and $\kappa_\mP(x) = \|x\|_1$.  Black curves show P-CGM with varying $\lambda$. Red curves show RP-CGM with fixed $\lambda$ at a small value, and sweep $\theta$ and $\bar \xi$.}
    \label{fig:sensing_smallm}
\end{figure}

\section{Conclusion}
\label{sec:conclusion}
This work considers two variations of the conditional gradient method (CGM): the P-CGM, which  accommodates gauge-based penalties in place of constraints, and the RP-CGM, which allows concave transformations of the gauges. 
The gauges may be induced by compact sets, but  also accomodate ``simple" directions of recession. 
We give a convergence rate to a stationary point, and propose a  gradient screening rule and support recovery guarantee. 
Compared with proximal methods, these CGM-based methods
often have a much cheaper  per-iteration cost; e.g. in the group norm,  computing the LMO (without reweighting) is trivial compared to even computing the gauge function itself. Additionally, the almost-for-free computation of the gap and residual quantity makes  screening a very small computational addition. 

The key challenge in showing the convergence of these methods is controlling the size of each $s^{(t)}$. This was trivial in the CGM case when $s^{(t)}$ was constrained in a compact set; when transformed to a penalty, we require a minimum amount of curvature of $\phi$ at $\xi\to+\infty$, and we restrict $\gamma$ to only having strict concavity over a finite support. However, as shown in the numerical results, these restrictions do not greatly inhibit the sparsifying effects of the penalty functions.


Finally, we do not incorporate away step  \citep{guelat1986some, lacostejulienjaggi}. In implementation, they are somewhat orthogonal to the extensions provided in this work, and can be added somewhat automatically; the analysis is a subject for future work.

\appendix

\section{Smoothness equivalences}
\label{app:smoothness}

\begin{lemma}[Smoothness equivalences]
Suppose that  $f$ is $L$-smooth with respect to $\mP$.
Then the following also holds:
\begin{enumerate}
\item Expansiveness
\begin{eqnarray}
(\nabla f(x)-\nabla f(y))^T(x-y)&\geq&   \frac{1}{2L}(\sigma_\mP(\nabla f(x)-\nabla f(y))^2\nonumber\\
&&\qquad + \sigma_\mP(\nabla f(y)-\nabla f(x))^2),
\label{eq:expansiveness}
\end{eqnarray}
\item Strongly convex conjugate
\begin{eqnarray}
f(y)-f(x) &\geq& \nabla f(x)^T(y-x) +  \frac{1}{2L}\sigma_\mP(\nabla f(y)-\nabla f(x))^2.
\label{eq:stronglyconvex}
\end{eqnarray}
\end{enumerate}
\label{lem:smoothness}
\end{lemma}
Since the proof is very similar to those presented in \cite{nesterovlectures}, we include them in appendix \ref{app:smoothness}.

\begin{proof}
The proof largely follows from \cite{nesterovlectures}, mildly adapted.
\begin{itemize}
\item First prove \eqref{eq:ass:smoothness} $\Rightarrow$  \eqref{eq:expansiveness}.
Construct $g(x) = f(x) - x^T\nabla f(y)$, which is convex, also $L$-smooth, and has minimum at $x = y$.
Then, for any $w$, 
\[
g(y) \leq g(x+w) \overset{(a)}{\leq} g(x) + \nabla g(x)^Tw + \frac{L}{2}\kappa_\mP(w)^2,
\]
where (a) is since $g$ is $L$ smooth and convex.

Now pick  
\[
w \in \frac{1}{L}\sigma_\mP(-\nabla g(x))\partial \sigma_\mP(-\nabla g(x)),
\]
which implies
\begin{eqnarray*}
\frac{L}{\sigma_\mP(-\nabla g(x))}  w & \in& \argmax{\kappa_\mP(u)\leq 1}\;\langle u, -\nabla g(x)\rangle = \partial \sigma_\mP(-\nabla g(x)),
\end{eqnarray*}
and thus
\[
\kappa_\mP(w) = \frac{\sigma_\mP(-\nabla g(x))}{L},
\]
and
\[
\langle w, -\nabla g(x)\rangle = \frac{1}{L}\sigma_\mP(-\nabla g(x))^2.
\]
Then
\[
\frac{L}{2}\kappa_\mP(w)^2 = \frac{1}{2L} \sigma_\mP(-\nabla g(x))^2,
\]
and
plugging in the construction for $g$ gives
\begin{eqnarray*}
g(y) -g(x)&\leq&  \underbrace{\nabla g(x)^Tw + \frac{L}{2}\kappa_\mP(w)^2}_{-\frac{1}{2L}\sigma_\mP(-\nabla g(x))^2}\\
\iff f(y) -f(x) &\leq&  (y-x)^T\nabla f(y)  - \frac{1}{2L} \sigma_\mP(\nabla f(y)-\nabla f(x))^2.
\end{eqnarray*}
Applying the last inequality twice gives
\begin{eqnarray*}
(y-x)^T(\nabla f(y)-\nabla f(x))\leq  \frac{1}{2L} ((\sigma_\mP(\nabla f(x)-\nabla f(y))^2  +(\sigma_\mP(\nabla f(y)-\nabla f(x))^2).
\end{eqnarray*}

\item Now prove \eqref{eq:ass:smoothness} $\Rightarrow$  \eqref{eq:stronglyconvex}.
Using the same $g$ as before, consider
\[
\min_z\; g(x)+\langle \nabla g(x), z-x\rangle + \frac{L}{2}\kappa_\mP(x-z)^2
=\min_w\; \langle \nabla g(x), w\rangle + \frac{L}{2}\kappa_\mP(w)^2.
\]
Using optimality conditions, picking $w = z-y$, we have
\[
0\in \nabla g(x) + L \kappa_\mP(w)\partial \kappa_\mP(w)\iff -\frac{1}{L\kappa_\mP(w)}\nabla g(x) = \argmax{\sigma_\mP(u)\leq 1}\langle u,w\rangle ,
\]
which implies
\[
\sigma_\mP(-\nabla g(x)) = L\kappa_\mP(w), \qquad  -\frac{1}{L\kappa_\mP(w)}\langle w,\nabla g(x)\rangle = \kappa_\mP(w).
\]
so
\[
\langle w, -\nabla g(x)\rangle = L\kappa_\mP(w)^2 = \frac{1}{L}\sigma_\mP(-\nabla g(x))^2,
\]
and overall
\[
g(y) \geq \min_z\; g(x)+\langle \nabla g(x), z-x\rangle + \frac{L}{2}\kappa_\mP(x-z)^2 = g(x) - \frac{1}{2L}\sigma_\mP(-\nabla g(x))^2.
\]
Plugging in $f$ gives
\[
f(y)-f(x) \geq (y-x)^T\nabla f(y) - \frac{1}{2L}\sigma_\mP(\nabla f(y)-\nabla f(x))^2.
\]
\end{itemize}
\end{proof}

\begin{customprop}{\ref{cor:uniquegrad-convex}}[Uniqueness of gradient] If \eqref{eq:ass:smoothness} holds and $0\in\inte~\mP$, then $\nabla f(x)$ is unique at the optimum.
\end{customprop}

\begin{proof}
Assume that $f(x)=f(x^*)$ for some $x\neq x^*$, $x$ feasible. Then by optimality conditions, \\
${\nabla f(x^*)^T(x^*-x)\leq 0}$, and thus
\[
\underbrace{f(x)-f(x^*)}_{=0} \geq \underbrace{\nabla f(x^*)^T(x-x^*)}_{\geq 0} + \frac{1}{2L}\sigma_\mP(\nabla f(x)-\nabla f(x^*))^2,
\]
which implies that $\sigma_\mP(\nabla f(x)-\nabla f(x^*)) = 0$. Since $0\in \inte~\mP$, this can only happen if $\nabla f(x) = \nabla f(x^*)$.
\qed
\end{proof}

\section{Convergence results from section \ref{sec:nonconvex}}
\label{app:convergence}

\begin{customprop}{\ref{prop:gapres-general}}[Residual]
Denoting  $\res_\mP(x) = \gap_\mP(x,-\nabla f(x);x)$ the gap at $x$ with reference $x$, then
 \[
 \res_\mP(x)  \geq 0 \; \forall x, \qquad \res_\mP(x) =0 \iff \text{ $x$ is a stationary point of \ref{eq:main-1norm}.}
 \]
\end{customprop}

\begin{proof} 
Denote $y = \argmin{y}\;f(x+y)$, and $z = -\nabla f(x+y)$, and plug in $\kappa_{\mP(x)}(x)=\bar r_\mP(x;x)$.
Then
\begin{eqnarray*}
\res_\mP(x)&=& f(x+y) + f^*(-z) +  \phi(r_\mP(x)) +  \phi^*(\sigma_{\mP(x)}(z))+ (\kappa_{\mP(x)}(x)- r_\mP(x))\cdot \sigma_{\mP(x)}(z)\\
&\overset{(a)}{=}& x^T\nabla f(x+y) + \phi( r_\mP(x)) +  \phi^*(\sigma_{\mP(x)}(z))+ (\kappa_{\mP(x)}(x)- r_\mP(x))\cdot \sigma_{\mP(x)}(z)\\
&\overset{(b)}{\geq}& x^T\nabla f(x+y) + \underbrace{y^T\nabla f(x+y)}_{\geq 0} + r_\mP(x)\sigma_{\mP(x)}(z) + (\kappa_{\mP(x)}(x)- r_\mP(x))\cdot \sigma_{\mP(x)}(z)\\
&\overset{(c)}{\geq} & x^T\nabla f(x+y) +  \kappa_{\mP(x)}(x)\cdot \sigma_{\mP(x)}(z) \\
&\overset{(d)}{\geq} & x^T\nabla f(x+y) -x^T\nabla f(x+y) = 0
\end{eqnarray*}
where
\begin{itemize}
    \item[(a)] uses the Fenchel-Young inequality on $f$ and $f^*$, 
    \item[(b)] uses the Fenchel-Young inequality on $\phi$ and $\phi^*$,
    \item[(c)] follows since $-\nabla f(x+y)\in\mK^\circ$ and $y\in \mK$, and thus $y^Tz \geq 0$, and
    \item[(d)] follows from the definition of $\sigma_{\mP(x)}$.
\end{itemize}
Tightness of (b) occurs iff Fenchel-Young is satisfied with equality, e.g. 
\begin{equation}
\sigma_{\mP(x)}(z) = \phi'(r_\mP(x)) 
\label{eq:optcond-helper1-app}
\end{equation}
Tightness of (c) occurs iff 
\begin{equation}
\sigma_{\mP(x)}(z)  = \frac{-\nabla f(x)^Tp}{\gamma'(\coeff_\mP(x;p))}, \quad \forall p, \; c_p \neq 0.
\label{eq:optcond-helper2-app}
\end{equation}
The ``element-wise" optimality conditions for \eqref{eq:main-nonconvex} are, for all $p\in \mP_0$,
\begin{eqnarray*}
\frac{-\nabla f(x)^Tp}{\gamma'(\coeff_\mP(x;p))} = \phi'(r_\mP(x)) \gamma'(c_p)  &\text{ if } & c_p \neq 0\\
\frac{-\nabla f(x)^Tp}{\gamma'(\coeff_\mP(x;p))}\leq \phi'(r_\mP(x)) \gamma'(c_p)&  \text{ if }&  c_p  = 0\\
\end{eqnarray*}
which is true iff \eqref{eq:optcond-helper1-app}, \eqref{eq:optcond-helper2-app}  hold.
\end{proof}

\begin{property}[Linearized objective value bound]
\label{prop:objvalbound}
Given assumptions \ref{asspt:phigamma1}, \ref{asspt:phi2},   \ref{asspt:atoms}, \ref{asspt:gensmooth}, then the objective error of each linearized problem decreases as
\[
\Delta^{(t)}  = O(1/t).
\]
\end{property}
\begin{proof}
Define 
\[
A = \left(\frac{6L \gamma_{\max}^2}{\mu^2\gamma_{\min}^2 }\sigma_{\widetilde\mP}(-\nabla f(x^*+y^*)+6\gamma_{\max}\nu_0+3\kappa_\mP(x^{(0)}) \right)^2, \qquad B = \frac{3L^2\gamma_{\max}^2}{\mu^{2}\gamma_{\min}^2}.
\]
Then putting together lemmas \ref{lem:onestepdescend1}, \ref{lem:onestepdescend2} and using the relation $(a+b)^2 \leq 2a^2 + 2b^2$ gives
\[
\Delta^{(t+1)}-\Delta^{(t)} \leq -\theta^{(t)}   \res_\mP(x^{(t)}) + (\theta^{(t)})^2\left( B\Delta^{(t)}  + B  \bar \Delta^{(t-1)} 
  + A \right) .
\]
where $\bar \Delta^{(t)}$ is defined as an averaging over square roots, e.g. 
\[
\sqrt{\bar\Delta^{(t)}} =\frac{2}{t(t+1)} \sum_{u=1}^{t} u \sqrt{\Delta^{(u)}}. 
\]
Then picking $\bar t > 6B$, we get that for all $t \geq \bar t$, $B(\theta^{(t)})^2 \leq \theta^{(t)}/3$, and therefore 
\begin{eqnarray*}
\Delta^{(t+1)}-\Delta^{(t)} &\leq& -\theta^{(t)}  \underbrace{ \res_\mP(x^{(t)})}_{\geq \Delta^{(t)}} + (\theta^{(t)})^2\left( B\Delta^{(t)}  + B  \bar \Delta^{(t-1)} 
  + A \right) \\
  &\leq& -\theta^{(t)} \Delta^{(t)}+ (\theta^{(t)})^2\left( B\Delta^{(t)}  + B\bar \Delta^{(t-1)} 
  + A \right) \\
  &\leq& -\frac{2\theta^{(t)} \Delta^{(t)}}{3} + (\theta^{(t)})^2\left( B  \bar \Delta^{(t-1)}
  + A \right) .
\end{eqnarray*}

We now pick $G$ large enough such that for all $t \leq \bar t$, $\Delta^{(t)} \leq G/t$, and $G > 24A$. Since $\Delta^{(t)}$ is always a bounded quantity ($x^{(t)}$ is always feasible), this is always possible. Then, for all $t < \bar t$, 
\[
\sqrt{\bar \Delta^{(t)}} \leq \frac{\sqrt{G}}{t(t+1)}\sum_{t'=1}^t \sqrt{t'}\overset{(a)}{\leq} \frac{2\sqrt{G}}{3t(t+1)}t^{3/2},
\]
where (a) is by  integral rule, and so
\[
\bar \Delta^{(t)} \leq  \frac{4Gt}{9(t+1)^2} \leq \frac{G}{2t}.
\]
Now we make an inductive step. Suppose that for some $t$, $\Delta^{(t')} < G/t'$ for all $t' \leq t$.  Then 
\begin{eqnarray*}
\Delta^{(t+1)} &\leq& \Delta^{(t)}-\frac{2}{3}\theta^{(t)}\Delta^{(t)} + (\theta^{(t)})^2(A+B\bar\Delta^{(t)})\\
&\leq & \frac{G}{t} - \frac{2}{3} \frac{2G}{t+1}\frac{1}{t} + \frac{4}{(t+1)^2}\left(A + \frac{GB}{2t}\right)\\
&= & \frac{G}{t+1}\left(\frac{t+1}{t} - \frac{4}{3t} + \frac{4A}{(t+1)G} + \frac{2B}{t(t+1)} \right)\\
&\leq & \frac{G}{t+1}\left(1 - \frac{1}{3t}  + \frac{4A}{tG} + \frac{2B}{t^2} \right)\\
&=& \frac{G}{t+1}\left(1 + \frac{1}{t} \left(-\frac{1}{3} + \underbrace{ \frac{4A}{G} }_{<1/6}+ \underbrace{\frac{2B}{t}}_{< 1/6}\right) \right) \leq \frac{G}{t+1},
\end{eqnarray*}
which satisfies the inductive step.
\end{proof}

The following is a generalized and modified version of a proof segment from \cite{jaggi2013revisiting}, which will be used for proving $O(1/t)$ gap convergence.
\begin{lemma}
\label{lem:gaphelper1b}
Pick some $0 < T_2 < T_1$ and pick  
\[
\bar k = \ceil{D(k+D)/(D+T_1)}-D \;\Rightarrow\;  \frac{D}{D+T_1} \leq \frac{\bar k+D}{k+D} \leq \frac{D}{D+T_2}.
\]
Then if
\[
 \frac{C_1(D+T_1)}{D} \leq C_3\cdot \log\left(\frac{D+T_2}{D}\right),
 \]
  then for all $k > T_1$,
\[
 \left(\frac{C_1}{D+\bar k} + \sum_{i=\bar k}^{k} \frac{C_2}{(D+i)^2}-\frac{C_3}{D+i}\cdot \frac{1}{D+k}\right)  < 0.
\]
\end{lemma}
\begin{proof}
Using integral rule, we see that
\[
\sum_{i=\bar k}^{k} \frac{1}{(D+i)^2}\leq  \int_{z=\bar k -1}^{k-1} \frac{1}{(D+i)^2} = \frac{1}{D-1+ k}-\frac{1}{D-1+\bar k}
\]

\[
\sum_{i=\bar k}^{k} \frac{1}{D+i}\geq  \int_{z=\bar k }^{k} \frac{1}{D+i} = \log(D+k)-\log(D+\bar k).
\]
This yields 
\begin{eqnarray*}
c(k) &:=&  \frac{C_1}{D+\bar k} + \sum_{i=\bar k}^{k} \frac{C_2}{(D+i)^2}-\frac{1}{D+i}\cdot \frac{C_3}{D+k}\\
&\leq &  \frac{C_1}{D+\bar k} + \frac{C_2}{D-1+ k}-\frac{C_2}{D-1+\bar k} + \frac{C_3}{D+k}\cdot (\log(D+\bar k) - \log(D+k))\\
&\leq &  \frac{C_1(D+T_1)}{D(D+k)} + \underbrace{\frac{C_2}{D-1+ k}-\frac{C_2}{D-1+\bar k}}_{<0} + \frac{C_3}{D+k}\cdot \log\left(\frac{D}{D+T_2}\right)\\
&\leq &  \frac{C_1(D+T_1)}{D(D+k)} + \frac{C_3}{D+k}\cdot \log\left(\frac{D}{D+T_2}\right)< 0.
\end{eqnarray*}
\end{proof}

\begin{lemma}
[Generalized non-monotonic gap bound]
\label{lem:gapbound}
Given 
\begin{itemize}
\item 
$\Delta^{(t)}  \leq \frac{G_1}{t+D}$ for some $G_1$, 
\item $\theta^{(t)} = \frac{G_2}{t+D}$ for some $G_2$ and $D$, and 
\item $\Delta^{(t+1)} - \Delta^{(t)}(1+\alpha \theta^{(t)}) \leq -\theta^{(t)}\res(x^{(t)}) + (\theta^{(t)})^2 G_3$ for some $G_3$,
\end{itemize}
then for 
\[
G_4 \geq \frac{G_1}{G_2}  \frac{(D+2)}{D( \log\left(\frac{D+1}{D}\right) )} ,
\]
we have
\[
\min_{i\leq t}\res(x^{(i)}) \leq \frac{G_4}{t+D}.
\]
\end{lemma}
\begin{proof}  
 We have
\[
\Delta^{(t+1)}-\Delta^{(t)}\leq \alpha \theta^{(t)}\Delta^{(t)} - \theta^{(t)}\gap^{(t)} + G_3(\theta^{(t)})^2.
\] 
Now assume that for all $i \leq t$, $\gap^{(i)} > \frac{G_4}{t+D}$. Then, telescoping from $\bar t$ to $t$ gives 
\begin{eqnarray*}
\Delta^{(t+1)}  &\leq& \Delta^{(\bar t)} + \sum_{i=\bar t}^t \left( \alpha \theta^{(i)}\Delta^{(i)} - \theta^{(i)}\gap^{(i)} + G_3(\theta^{(i)})^2\right)\\
& < & \frac{G_1}{\bar t + D} + \sum_{i=\bar t}^t \left( \alpha \frac{G_1G_2}{(i+D)^2}- \frac{G_2}{i+D}\frac{G_4}{t+D} + \frac{G_3 G_2^2}{(i+D)^2}\right).
\end{eqnarray*}

Picking $C_1 = G_1$, $C_2 =\alpha G_1G_2+G_3G_2^2$, $C_3 = G_2G_4$, and invoking Lemma \ref{lem:gaphelper1b}, this yields that $\Delta^{(t+1)} < 0$, which is impossible. Therefore, the assumption must not be true. 
\end{proof}

Piecing everything in this section together gives theorem \ref{th:convergence}



\section{Screening proofs from section \ref{sec:nonconvex}}
\label{app:secnonconvex_proofs}

\begin{customprop}{\ref{prop:res-bound-graderr}}[Residual bound on gradient error]
Denote $D(x) = r_\mP(x)-r_\mP(x^*) + \bar r_\mP(x;x) - \bar r_\mP(x^*;x)$
the linearization error at $x$. 
Denoting $x^*$ a stationary point of \eqref{eq:main-nonconvex} and $y(x) = \argmin{y'\in \mK}\; f(x+y')$, then
\begin{multline*}
\sigma_{\widetilde\mP}(\nabla f(x+y(x))-\nabla f(x^*+y(x^*))) \leq\\ \frac{LD(x)}{2\gamma_{\min}} + \sqrt{\frac{L^2D(x)^2}{4\gamma^2_{\min}} + L\res(x) + LD(x) \frac{\sigma_{\widetilde\mP}(\nabla f(x+y(x)))}{\gamma_{\min}}}.
\end{multline*}
\end{customprop}

\begin{proof}
First, note that 
\begin{eqnarray}
\phi^*(\sigma_{\mP(x)}(z)) + r_0\cdot (\sigma_{\mP(x)}(z))  &=&  \sup_y\; y^Tz - \phi(r_0 + \kappa_{\mP(x)}(y)) \nonumber\\
&\geq &   z^Tx^* - \phi(r_0 + \kappa_{\mP(x)}(x^*)). \label{eq:phistar-ell1-proofhelper1b}
\end{eqnarray}
Define $\res(x) = (\bar F(x;x)-\bar F_D(-\nabla f(x);x)$.
Taking $(x,-\nabla f(x))$ as a feasible primal-dual pair and reference point $\bar x=x$, and denoting $\epsilon(x) = \phi(r_0+\kappa_{\mP(x)}(x^*)) - \phi(x^*)$, $z = -\nabla f(x+y(x))$, and $z^* = -\nabla f(x^* + y(x^*))$, then
\begin{eqnarray*}
\res(x) &=&\underbrace{f(x) + f^*(z)}_{\text{use Fenchel-Young}}+\phi(r(x)) \\
&&\qquad + \underbrace{\phi^*(\sigma_{\mP(x)}(z) -r_0 \cdot (\sigma_{\mP(x)}(z)) }_{\text{use \eqref{eq:phistar-ell1-proofhelper1b}}} \\
&\geq & -z^T(x- x^*) + \phi(r_\mP(x))   - \phi\left(r_0 + \kappa_{\mP(x)}(x^*)\right)\\
&\overset{+\epsilon(x)-\epsilon(x)}{\geq}  & -z^T(x- x^*) + \underbrace{\phi(r_\mP(x)) -\phi(r_\mP(x^*))}_{\text{convex in $x$}}- \epsilon(x)\\
&\overset{g\in \partial h(x^*)}{\geq} & -z^T(x- x^*) + g^T(x- x^*) - \epsilon(x).
\end{eqnarray*}
Picking in particular $g = -\nabla f( x^* + y(x^*))$,
\[
\res(x) +\epsilon(x) \geq  (x- x^*)^T(z^*-z) 
\overset{(\star)}{\geq} \frac{1}{L}\sigma_{\widetilde\mP}(z-z^*)^2
\]
where $(\star)$ follows from assumption \ref{asspt:gensmooth}.


Next, note that
\begin{eqnarray*}
\epsilon(x) &=& \phi(r_\mP(x)-r_\mP(x;x)+r_\mP(x^*;x))-\phi(r_\mP(x^*))\\
&\overset{\text{convex $\phi$}}{\leq} & \phi'(r_\mP(x^*))\underbrace{(r_\mP(x)-r_\mP(x;x)+r_\mP(x^*;x)-r_\mP(x^*))}_{=: D(x)}
\end{eqnarray*}
where in general,  $D(x) \leq (\gamma_{\max}-\gamma_{\min})\kappa_{\widetilde\mP}(x-x^*)$ and $D(x) = 0$ if $\gamma(\xi) = \xi$ (convex case).
Noting that, at optimality, 
\[
\phi'(r_\mP(x^*)) = \sigma_{\mP(x^*)}(z^*) \leq \frac{\sigma_{\widetilde\mP}(z^*)}{\gamma_{\min}},
\]
then
\[
\gamma_{\min}\phi'(r(x^*)) \leq \sigma_{\widetilde\mP}(z^*) \leq  \sigma_{\widetilde\mP}(z) + \sigma_{\widetilde\mP}(z-z^*) 
\]
and overall,
\begin{eqnarray*}
\sigma_{\widetilde\mP}(z^*-z)^2 &\leq& L \res(x) + L\epsilon(x) \\
&\leq& L \res(x) + LD(x)\frac{  \sigma_{\widetilde\mP}(z) + \sigma_{\widetilde\mP}(z^*-z) }{\gamma_{\min}}.
\end{eqnarray*}
This inequality is quadratic in $\sigma_{\widetilde\mP}(z^*-z)$, which leads to the bound 
\[
\sigma_{\widetilde\mP}(z^*-z)\leq \frac{LD(x)}{2\gamma_{\min}} + \sqrt{\frac{L^2D(x)^2}{4\gamma^2_{\min}} + L\res(x) + LD(x) \frac{\sigma_{\widetilde\mP}(z)}{\gamma_{\min}}}.
\]
\qed
\end{proof}

\section{Extra experiments}
\label{app:experiments}

\subsection{Least squares synthetic experiment}
We generate the problems using the same parameters as given in section \ref{sec:experiments}.

\paragraph{Convergence.} Figures \ref{fig:convergence_convex} and \ref{fig:convergence_nonconvex} show the convergence and screening behavior of P-CGM and RP-CGM for varying parameters. We take $m = n = 100$ and generate $(x_0)_i\sim \mN(0,1)$ i.i.d. We take $\phi(\xi) = p^{-1} \xi^p$ and $\kappa_\mP(x) = \|x\|_1$.  In the nonconvex case, we pick $\gamma$ as the piecewise smooth function:
\begin{equation}
\gamma_{\mathrm{LSP}}(\xi) = \log\left(1+\frac{\xi}{\theta}\right), \qquad \gamma(\xi) = 
\begin{cases}
 \gamma_{\mathrm{LSP}}(\xi) & \text{ if }\xi \leq \bar \xi\\
(\xi - \bar \xi) \gamma_{\mathrm{LSP}}'(\bar \xi) +  \gamma_{\mathrm{LSP}}(\bar \xi)& \text{ else. }\\
\end{cases}
\label{eq:defgamma:exp}
\end{equation}

\begin{figure}
    \centering
    \includegraphics[width=2.25in]{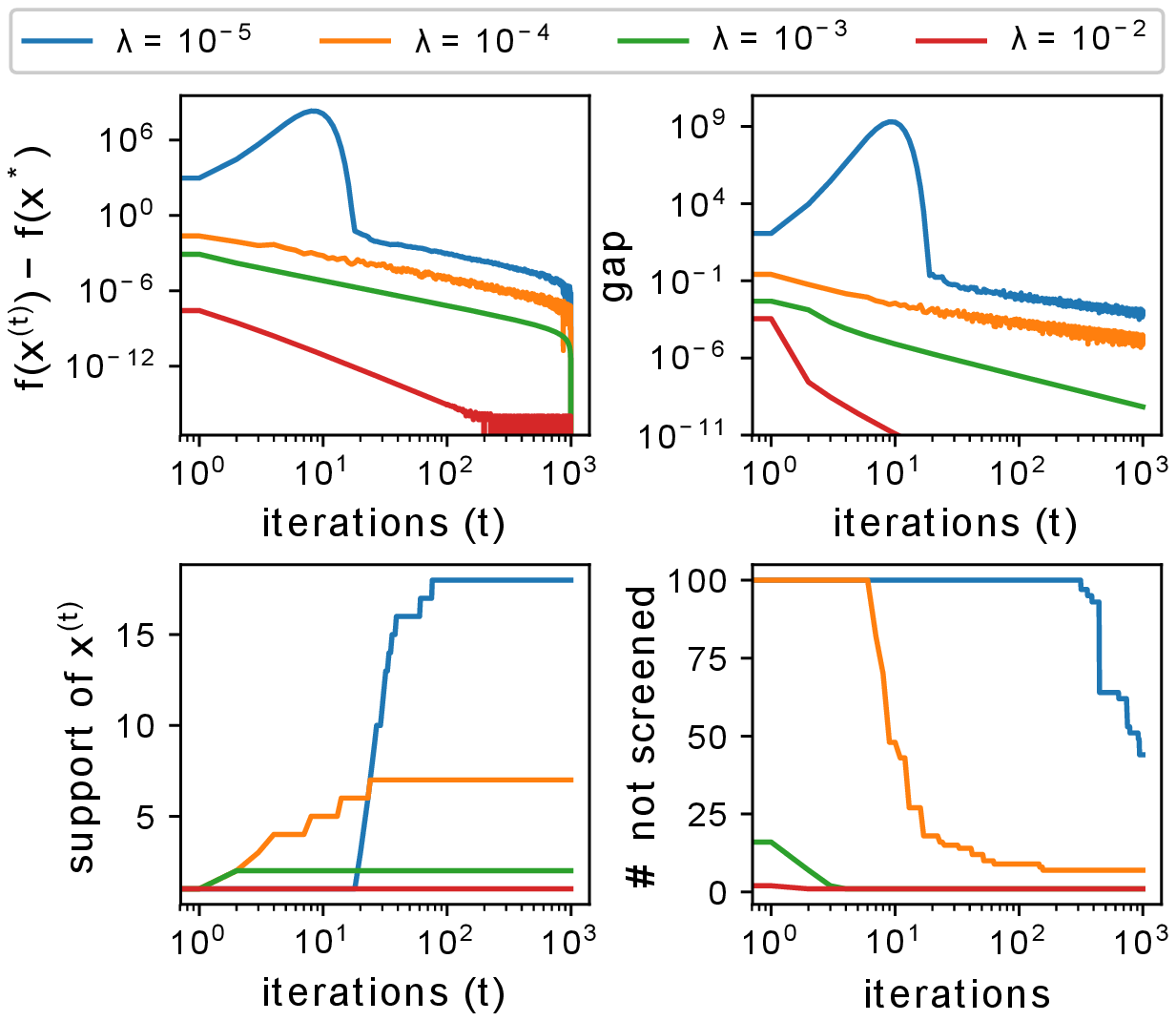}
    \includegraphics[width=2.25in]{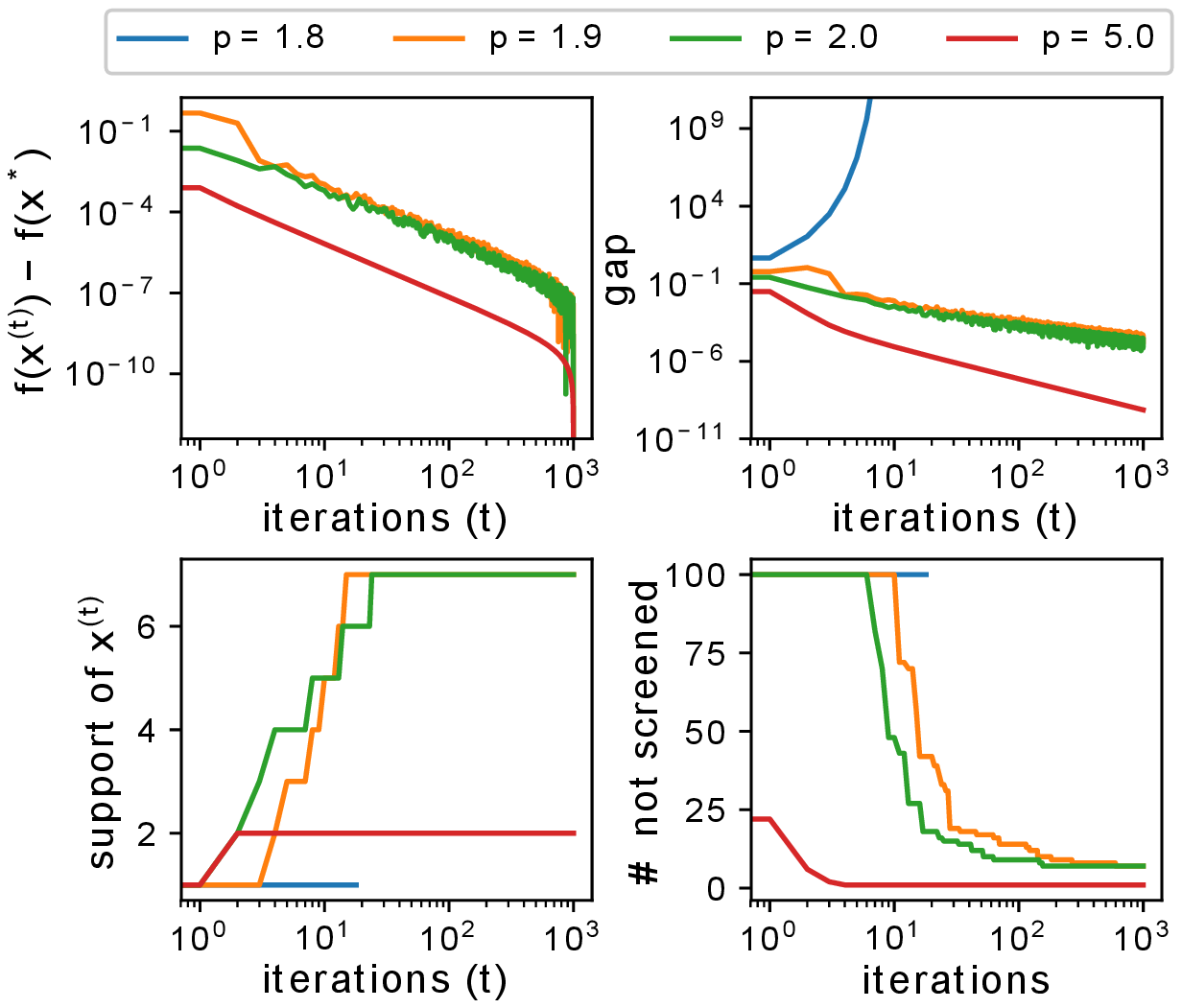}
    \caption{Convergence behavior of convex penalty with varying $\lambda$ and $p$, where $\phi(\xi) = p^{-1} \xi^p$ and $\kappa_\mP(x) = \|x\|_1$. Values when not swept: $p = 2$, $\lambda = 0.0001$.
    As suspected, $\lambda$ has an impact not just on the sparsity of $x^*$, but larger values lead to faster convergence and screening time (by constant factors).
    The convergence proof itself requires $p \geq 2$, and indeed we see divergence at $p = 1.75$, but at $p = 1.9$ the method still converges, showing the requirement is not necessarily tight. }
    \label{fig:convergence_convex}
\end{figure}
\begin{figure}
    \centering
    \includegraphics[width=2.25in]{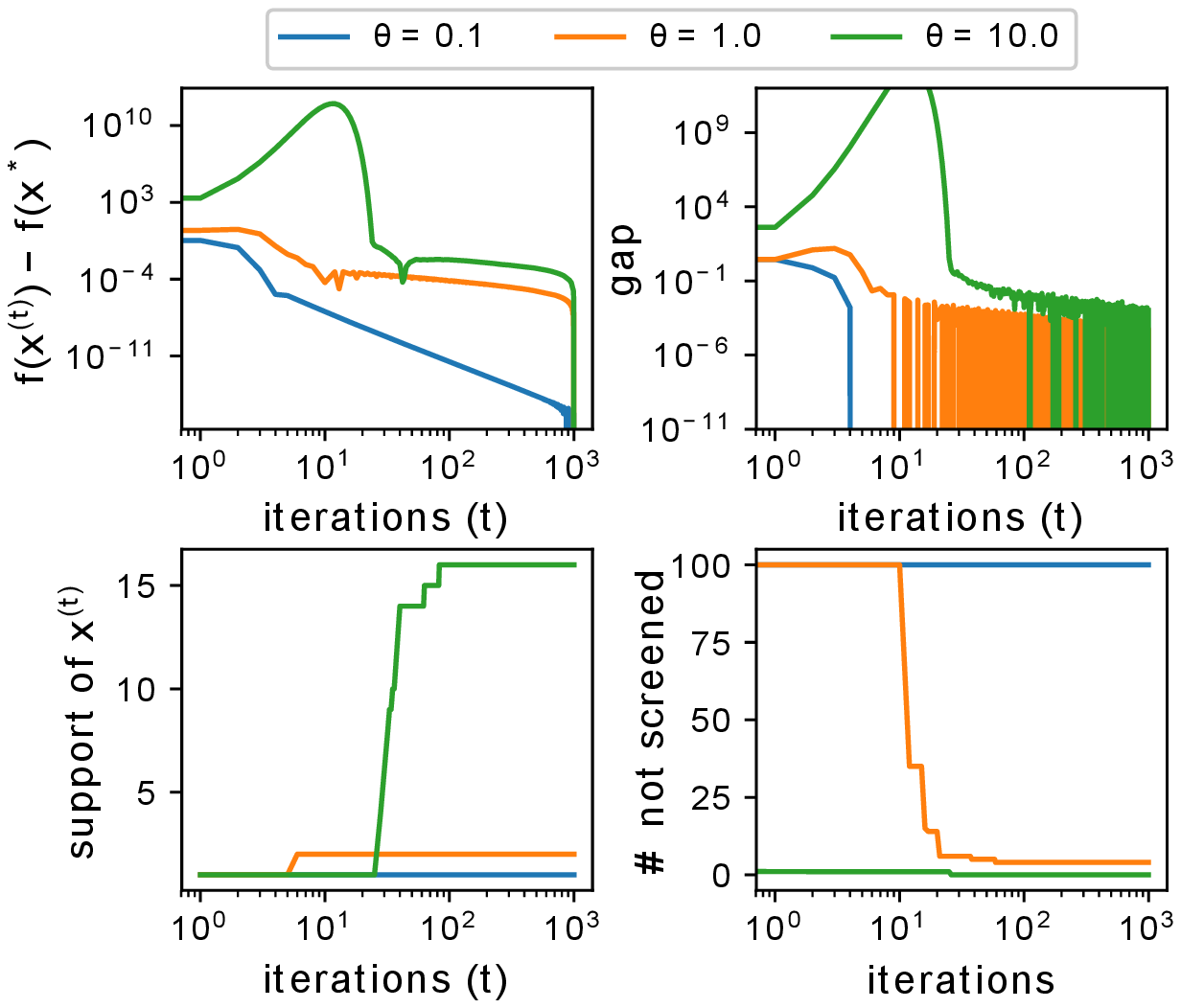}
    \includegraphics[width=2.25in]{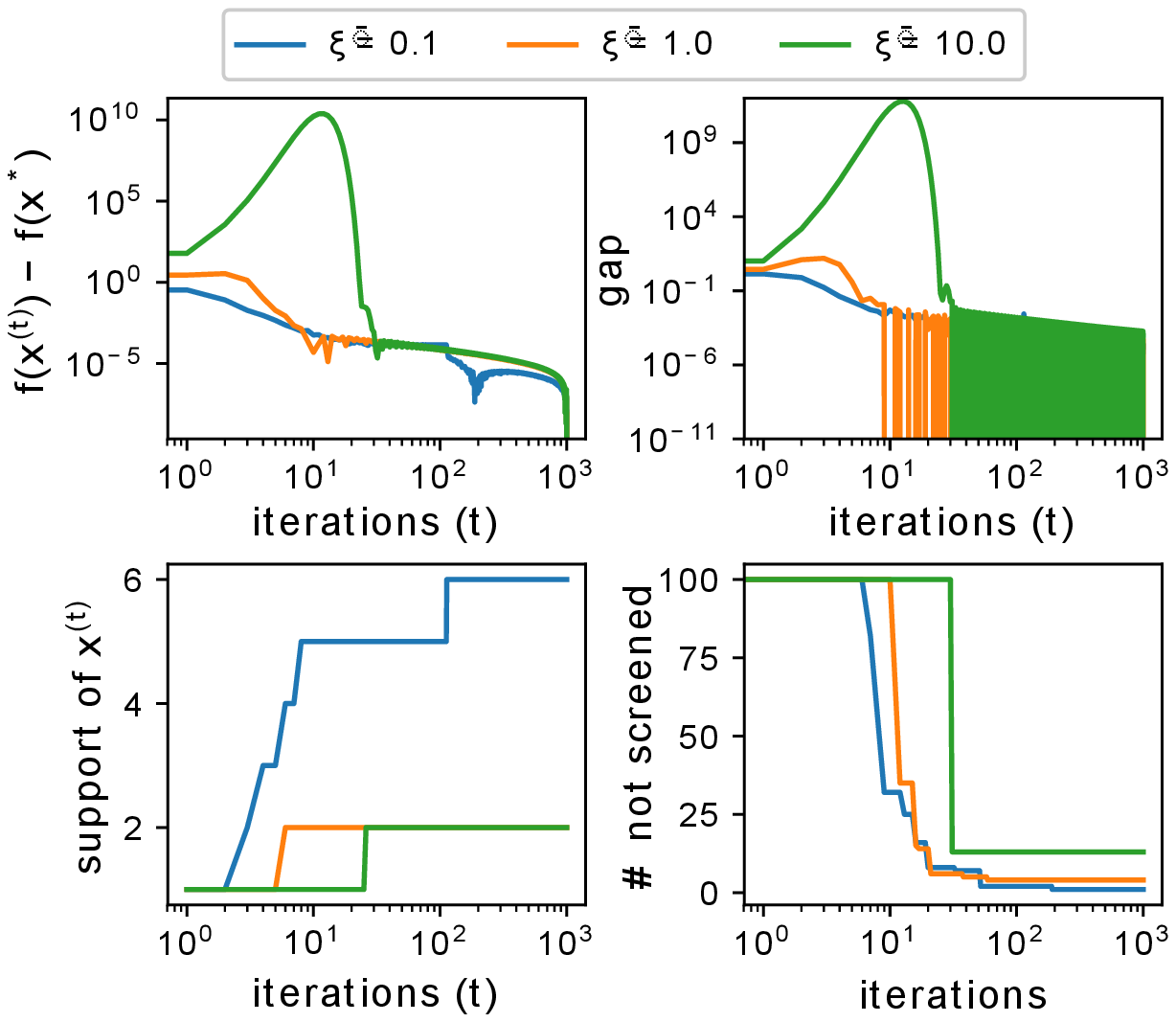}
    \caption{Convergence behavior using nonconvex LSP penalty \eqref{eq:defgamma:exp}, with varying $\theta$ and $\bar \xi$. Values when not swept: $p = 2$, $\lambda = 0.0001$, $\theta = 1$, $\bar \xi = 1$. Smaller $\theta$ and larger $\bar\xi$ leads to ``more concavity", and again more aggressive convergence and sparsity. Often it leads to more aggressive screening as well, but the pattern is not definitive (as indicated by bottom center left).}
    \label{fig:convergence_nonconvex}
\end{figure}

\paragraph{Gauges.} We now fix $\gamma(\xi) = \xi$ and consider different gauge penalties; specifically, the $\ell_1$ norm (fig. \ref{fig:sensing:ell1}), the TV norm (fig. \ref{fig:sensing:tv})
, and the latent overlapping group norm (fig. \ref{fig:sensing:group}), visualizing all the players: variable $x$, gradient $\nabla f(x)$, atom weights (sparse), and dual atom weights (screened).

\begin{figure}
    \centering
    \includegraphics[width=4.5in]{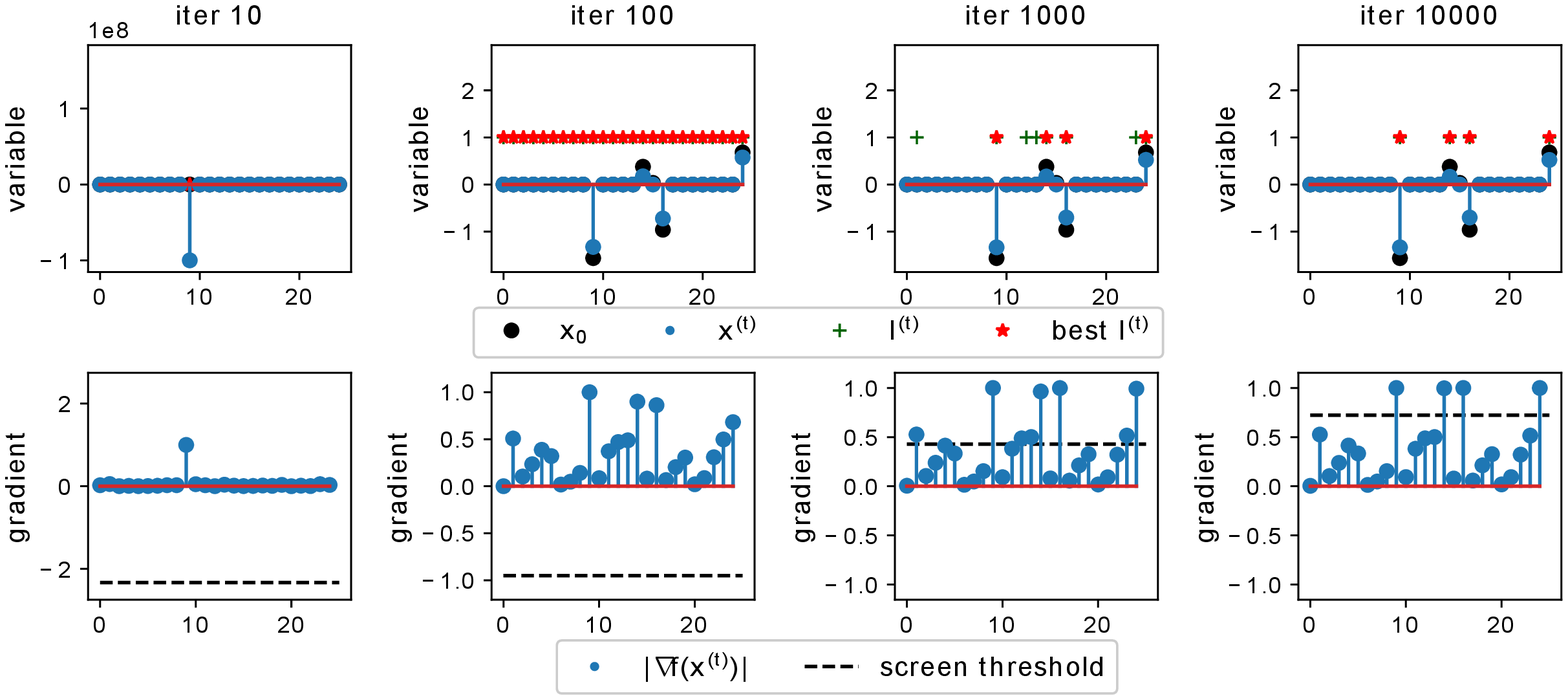}
    \caption{\textbf{Sparse sensing.} Top row shows $x^{(t)}$, $x^{(0)}$ the ground truth, $\mI^{(t)}$ the currently screened set, and the best $\mI^{(t)}$ which is the cumulatively screened set. Bottom row shows the gradient absolute values and the screen thresholds (normalized), which describe $\mI^{(t)}$. The screening method (red dots) quickly approach the true support (black stems) as the residual (screen threshold) gets closer to the maximum gradient value.}
    \label{fig:sensing:ell1}
\end{figure}

\begin{figure}
    \centering
    \includegraphics[width=4.5in]{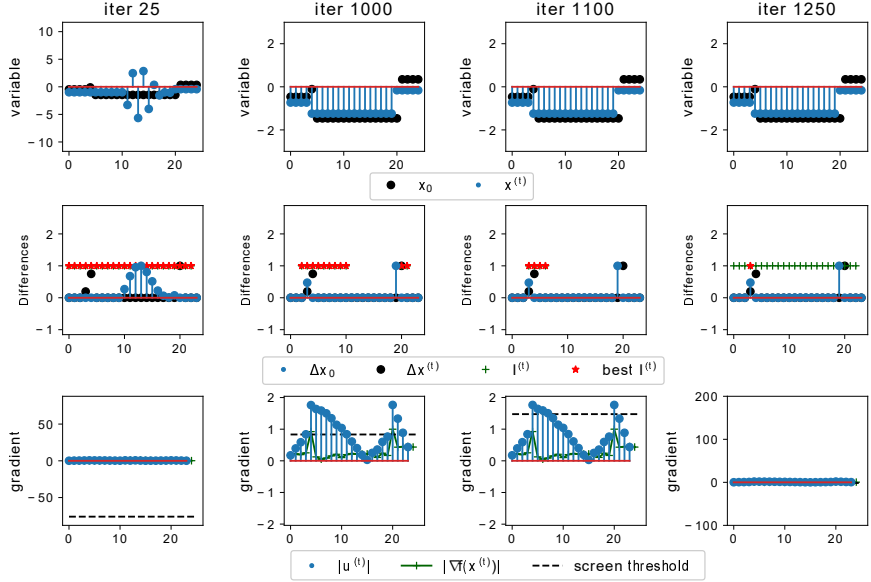}
    \caption{\textbf{1-D Smoothing.} Top row plots $x_0$ the peicewise smooth ground truth and $x^{(t)}$ at each iteration. Note the early iterations have significant jitter, which smooths out in later iterations. Second row shows $Dx_0$ and $Dx^{(t)}$ the difference vectors, which are the sparse elements. Third row plots $u^{(t)} = D^\dagger \nabla f(x^{(t)})$  along with $\nabla f(x^{(t)})$ (normalized); here, $u^{(t)}$ is the actual vector being screened, with peaks identifying edges in $x^{(t)}$. Although the structure of the atoms (3rd row) looks very different, the same screening principle applies. }
    \label{fig:sensing:tv}
\end{figure}

\begin{figure}
    \centering
    \includegraphics[width=4.5in]{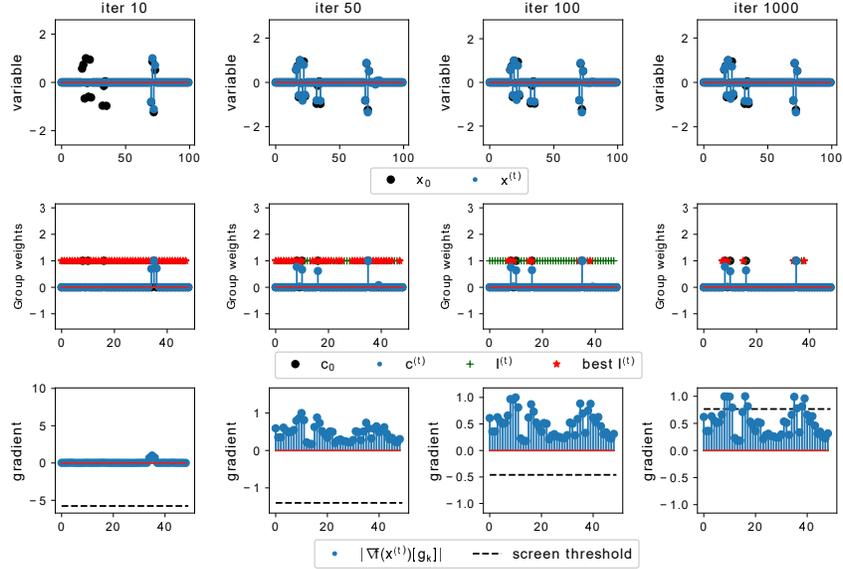}
    \caption{\textbf{Overlapping group sparsity.} The groups are given as  $g_i = \{2i,2i+1,2i+2,2i+3\}$ for $i = 1,...,n/2$. Top row plots $x_0$ (ground truth) and $x^{(t)}$. Second row shows the weights $c$ on each atom and screening variables; note that $c$ can only be determined after solving an optimization problem. Bottom row shows the dual weights $\|\nabla f(x^{(t)})_{g_i}\|_2$, which can be computed directly (no optimization needed) and are used to find the LMO and screen efficiently. The screening itself is erratic, but since it is safe, picking only the best $\mI^{(t)}$ at each iteration still quickly converges to the true support.
}
    \label{fig:sensing:group}
\end{figure}


\bibliographystyle{alpha}      

\end{document}